\documentclass[11pt,a4paper]{article}
\usepackage[T1]{fontenc}
\usepackage[utf8]{inputenc}
\usepackage{lmodern}
\usepackage{amsmath,amssymb,amsthm,mathtools}
\usepackage[margin=28mm]{geometry}
\usepackage{microtype}
\usepackage{enumitem}
\usepackage[colorlinks=true,linkcolor=blue,citecolor=blue,urlcolor=blue]{hyperref}
\usepackage{bookmark}
\numberwithin{equation}{section}
\newtheorem{theorem}{Theorem}[section]
\newtheorem{proposition}[theorem]{Proposition}
\newtheorem{lemma}[theorem]{Lemma}
\newtheorem{corollary}[theorem]{Corollary}
\theoremstyle{definition}
\newtheorem{definition}[theorem]{Definition}
\theoremstyle{remark}
\newtheorem{remark}[theorem]{Remark}
\newcommand{\R}{\mathbb{R}}
\newcommand{\Z}{\mathbb{Z}}
\newcommand{\T}{\mathbb{T}}

\DeclareMathOperator{\Corr}{Corr}

\DeclareMathOperator{\Lip}{Lip}

\setlist[enumerate]{label=\textup{(\roman*)},leftmargin=*,itemsep=3pt}
\hypersetup{
  pdftitle={A Smooth Autonomous Flow with Exponential Mixing on the Three-Torus},
  pdfauthor={Nianci Jiang and Weili Zhang}
}
\title{A Smooth Autonomous Flow with Exponential Mixing on the Three-Torus}
\author{
Nianci Jiang$^{1,2}$ and Weili Zhang$^{3}$\\[1em]
{\small $^{1}$School of Mathematical Sciences,
University of Chinese Academy of Sciences,
Beijing 100049, China}\\
{\small $^{2}$Academy of Mathematics and Systems Science,
Chinese Academy of Sciences,
Beijing 100190, China}\\
{\small $^{3}$Department of Mathematical Sciences,
Tsinghua University, Beijing 100084, China}\\[0.5em]
}
\date{}
\begin{document}
\maketitle
\begingroup
\renewcommand{\thefootnote}{}
\footnotetext{%
\begin{tabular}[t]{@{}l@{}}
\texttt{jiangnianci22@mails.ucas.ac.cn}\\
\texttt{zhangweili@amss.ac.cn}
\end{tabular}%
}
\endgroup

\vspace{-0.5cm}
\begin{abstract}
We construct a smooth autonomous divergence-free vector field on $\mathbb{T}^3$ whose flow is exponentially mixing with respect to Lebesgue measure. This answers an open question about smooth autonomous exponential mixing in three dimensions. Starting from an exponentially mixing Anosov suspension flow, we perform surgery near four periodic orbits and show that the resulting manifold is diffeomorphic to $\mathbb{T}^3$. The resulting vector field has exactly sixteen hyperbolic equilibria. We also prove a general transfer theorem that preserves exponential mixing under certain singular changes of coordinates.
\end{abstract}

\section{Introduction and main results}\label{sec:introduction}

Mixing by incompressible flows is a basic mechanism in fluid dynamics.
A velocity field transports and deforms material elements by repeated
stretching and folding, thereby transferring large-scale variations of a passive scalar to progressively finer spatial scales. From the viewpoint of dynamical systems, a standard way to quantify mixing is through the decay of correlations between observables.

Let $\varphi^t$ be a flow on a compact manifold $M$ preserving a probability measure $\mu$. For $f,g\in L^2(M,\mu)$, define their correlation at time $t$ by
\begin{equation}\label{eq:intro-correlation}
\operatorname{Corr}_{\mu,\varphi}(f,g;t)
:=
\int_M f(x)g(\varphi^t(x))\,d\mu(x)
-
\left(\int_M f\,d\mu\right)
\left(\int_M g\,d\mu\right).
\end{equation}
Its decay measures how rapidly correlations between observables are lost under the flow. Typical rates
include polynomial, stretched-exponential, and exponential decay. Exponential mixing is one of the strongest commonly studied quantitative mixing properties in smooth dynamics.

\begin{definition}\label{def:exponential-mixing}
Let $0<\alpha<1$. The flow $\varphi^t$ is called
\emph{exponentially mixing for $C^\alpha$ observables} if there exist
constants $C_\alpha,c_\alpha>0$ such that 
\begin{equation}\label{eq:intro-exponential}
\left|
\operatorname{Corr}_{\mu,\varphi}(f,g;t)
\right|
\le
C_\alpha e^{-c_\alpha t}
\|f\|_{C^\alpha(M)}
\|g\|_{C^\alpha(M)}
\end{equation}
for every $f,g\in C^\alpha(M)$ and every $t\ge0$.
\end{definition}

\begin{remark}\label{rem:regularity-formulations}
For \(h\in C^\alpha(M)\), we use the seminorm and norm 
\[
[h]_{C^\alpha}
:=
\sup_{x\neq y}
\frac{|h(x)-h(y)|}{d_{M}(x,y)^\alpha},
\qquad
\|h\|_{C^\alpha}
:=
\|h\|_{L^\infty}+[h]_{C^\alpha},
\]
where \(d_M\) denotes the Riemannian distance induced by a fixed smooth Riemannian metric on \(M\). Different regularity classes are used in the literature to formulate exponential mixing. We adopt the H\"older formulation used for Anosov flows in \cite{TZ}. In the passive-scalar literature, one often works instead with Sobolev observables. On $\mathbb T^3$ with Lebesgue measure, exponential correlation decay for H\"older observables implies exponential correlation decay for $H^s$ observables for every $s>0$, possibly with a smaller decay rate. We prove the version needed here in Appendix~\ref{app:regularity-transfer}.
\end{remark}
An interesting problem is whether exponential mixing can be achieved under strong regularity assumptions on the velocity field.

This question appears explicitly in several forms in the literature. In
\cite{ElgindiZlatos}, the authors asked whether one can construct an incompressible velocity field on the flat torus $\mathbb T^d$, $d\ge3$, which is smooth uniformly in time and whose flow is an exponential mixer. More recently, \cite{CoopermanIyerNolen} identified the existence of a \(C^1\), time-independent, divergence-free exponentially mixing vector field on \(\mathbb T^3\) as an open problem. The latter is an autonomous version of the problem and is the one addressed in this paper.

Several related constructions achieve exponential decay of mixing scales for suitably chosen initial data, or use velocity fields that may depend on
the initial datum \cite{AlbertiCrippaMazzucato,YaoZlatos}. Other results provide exponential mixing for a fixed velocity field, but do not satisfy all of the requirements above: existing examples either have lower spatial regularity \cite{ElgindiZlatos}, are deterministic but time-dependent \cite{ElgindiLissMattingly}, or are based on random time-dependent dynamics
\cite{BlumenthalCotiZelatiGvalani,Cooperman,CotiZelatiNavarroFernandez,JiangZhang,
Melzi,ZhangHamiltonian}.

The autonomous case is more restrictive. In dimension two, there are strong obstructions to exponential mixing for smooth area-preserving autonomous flows; see \cite{DKRH,Young}. In dimension three, smooth volume-preserving mixing flows are known \cite{HuPesinTalitskaya}, but this result does not provide exponential decay of correlations.

Consequently, the existence of a smooth autonomous divergence-free exponentially mixing flow on $\mathbb T^3$ remained open. The main result of this paper provides such a flow. To the best of our knowledge, this is the first construction.

\begin{theorem}\label{thm:main}
Let $m$ be the normalized Lebesgue measure on $\mathbb T^3$.
There exists a $C^\infty$ autonomous divergence-free vector field $X$ on $\mathbb T^3$ whose flow is exponentially mixing with respect to \(m\) for every \(0<\alpha<1\), in the sense of Definition~\ref{def:exponential-mixing}. Moreover, $X$ has exactly sixteen hyperbolic equilibria.
\end{theorem}

The theorem allows equilibria and therefore does not address the stronger problem of constructing a nowhere-vanishing autonomous exponential mixer.

The construction starts from an exponentially mixing Anosov suspension. We perform surgery near four periodic orbits by exchanging meridians and longitudes, and identify the resulting manifold with \(\mathbb T^3\). The vector field and its invariant volume extend smoothly across the surgery cores. Away from these cores, the new flow is conjugate to the seed flow with the same time parametrization, but the conjugacy becomes singular near the removed orbits. The main analytic difficulty is therefore to control the regularity of pulled-back observables. A cutoff argument balances this loss of regularity against the small volume of neighborhoods of the singular set, yielding a general quantitative transfer theorem for exponential mixing.

\begin{remark}\label{re:higher-dimensional tori}
The main theorem also gives examples on $\mathbb T^{3m}$ for every integer $m\geq 1$. Indeed, taking the product of $m$ copies of the constructed flow gives a smooth, autonomous, divergence-free vector field on $\mathbb T^{3m}$. This product flow is also exponentially mixing; see \cite[Theorem A.4]{DolgopyatFayadLiu2022} for the corresponding result for maps.

We are also investigating the remaining dimensions. Our approach uses torus extensions of the suspension flow, followed by the same surgery in the three-dimensional factor. The main additional difficulty is to establish exponential mixing for these extensions. We plan to address this in a separate paper.
\end{remark}

\paragraph{Notation.}
Unless otherwise stated, all manifolds are smooth, compact, connected, without boundary and are equipped with fixed smooth Riemannian metrics. We write $d_M$ for the corresponding Riemannian distance.
For a smooth map $F:M\to N$, the notation $F^*$ denotes the pullback of differential forms and $F_*$ denotes the pushforward of measures. When $F$ is a diffeomorphism, $F_*$ also denotes the pushforward of vector fields. A positive volume form $\Omega$ on $M$ determines a Borel measure $\mu_\Omega$ by $\mu_\Omega(A):=\int_A \Omega$
for every Borel set $A\subset M$. We use the same notation for the volume form and its associated measure whenever no confusion can arise. The notation $\cong$ denotes a smooth diffeomorphism of manifolds or an isomorphism of groups, according to context. All observables are taken to be real-valued.

\paragraph{Organization of the paper.}
Section~\ref{sec:seed} constructs the exponentially mixing suspension
that serves as the seed flow.
Section~\ref{sec:charts} introduces exact solid-torus coordinates near four selected periodic orbits, and Section~\ref{sec:surgery} performs the
volume-preserving surgery and describes the resulting equilibria.
Section~\ref{sec:topology} proves that the surgered manifold is diffeomorphic to $\mathbb T^3$.
Section~\ref{sec:transfer} states and proves a general quantitative transfer theorem for exponential mixing through a singular conjugacy.
Section~\ref{sec:completion} verifies its hypotheses for the present surgery and transports the resulting flow to $\mathbb{T}^3$ with normalized Lebesgue volume.

\section{An explicit exponentially mixing suspension}\label{sec:seed}

We first construct the seed flow used throughout the proof.
Starting from a hyperbolic automorphism, we form a smooth suspension flow with a suitably chosen roof function. We will prove that this flow preserves a smooth volume and is an exponentially mixing
Anosov flow.

The later sections modify this flow near four periodic orbits and show
that the resulting flow is defined on a manifold diffeomorphic to
\(\mathbb T^3\), while retaining exponential mixing.

\subsection{The base map and the roof function}

Set
\[
A=
\begin{pmatrix}
2&1\\
1&1
\end{pmatrix},
\qquad
B=A^3=
\begin{pmatrix}
13&8\\
8&5
\end{pmatrix},
\qquad
\widehat{\Sigma}:=\mathbb R^2/(2\mathbb Z)^2.
\]
Since \(B\in SL(2,\mathbb Z)\), the linear map
\(q\mapsto Bq\) preserves the lattice \((2\mathbb Z)^2\)
and hence descends to a smooth diffeomorphism
\[
B:\widehat{\Sigma}\longrightarrow\widehat{\Sigma},
\qquad
[q]\longmapsto[Bq].
\]
Define
\[
P:=\mathbb Z^2/(2\mathbb Z)^2
=
\{(0,0),(1,0),(0,1),(1,1)\}
\subset \widehat{\Sigma},
\qquad
q_*:=\left(\frac12,0\right)\in\widehat{\Sigma}.
\]
By direct computation, we have the following lemma.
\begin{lemma}\label{lem:seed-base}
The map \(B:\widehat{\Sigma}\to\widehat{\Sigma}\) is a smooth area-preserving diffeomorphism. Its eigenvalues are 
\begin{equation}
\Lambda=9+4\sqrt5>1,
\qquad
\Lambda^{-1}=9-4\sqrt5.
\label{eq:seed-eigenvalues}
\end{equation}
The map \(B\) fixes every point of \(P\), as well as the additional point \(q_*\notin P\).
\end{lemma}

Now we define the roof function. To retain integer translations as symmetries, we choose the roof function to be $\mathbb Z^2$-periodic, rather than merely $(2\mathbb Z)^2$-periodic. Let \begin{equation*}
a(t):=\begin{cases}e^{-1/t},&t>0,\\0,&t\leq0,\end{cases}
 \qquad
 \beta(t):=\frac{a(1-t)}{a(1-t)+a(t-1/4)}.
\end{equation*}
Then $\beta$ is smooth and satisfies
$\beta(t)=1$ for $t\le 1/4$ and $\beta(t)=0$ for $t \ge 1$.
Define the radial cutoff $b(v):=\beta(256|v|^2)$ for $v \in \mathbb R^2.$ Then $b(v)=1$ for $|v|\le 1/32$ and $b(v)=0$ for $|v|\ge 1/16$.

Finally, for $q\in\mathbb R^2$, set
\[
\chi(q):=
\sum_{k\in\mathbb Z^2}
b(q-q_*-k),
\qquad
\rho(q):=
1+(\sqrt2-1)\chi(q).
\label{eq:roof}
\]

\begin{lemma}\label{lem:seed-roof}
The function $\rho$ is smooth and $\mathbb Z^2$-periodic, and hence
descends to a smooth positive function on
$\widehat{\Sigma}$. Moreover,
\begin{equation}\label{eq:seed-roof-values}
 1\leq\rho\leq\sqrt2,\qquad
 \rho=1\text{ near }P,\qquad
 \rho=\sqrt2\text{ near }q_*+P.
\end{equation}
\end{lemma}

\begin{proof}
The function $b$ is a smooth bump function supported in $\overline{B(0,1/16)}$. 
Since the centers $q_*+k$ for $k\in\mathbb Z^2$ are mutually separated by distance at least $1$, their supports are disjoint. Hence the sum defining $\chi$ is locally finite and smooth, with $0\le\chi\le1.$ Reindexing the sum shows that $\chi$ is $\mathbb Z^2$-periodic, so $\rho$ is $\mathbb Z^2$-periodic. Since every integer point has distance at least $1/2$ from the set of centers, we have $\chi=0$ near every point of $\mathbb Z^2$, while $\chi=1$ near every point of $q_*+\mathbb Z^2$.
\end{proof}

\subsection{The suspension flow and its invariant volume}

Using the base map \(B:\widehat{\Sigma}\to\widehat{\Sigma}\)
and the roof function \(\rho:\widehat{\Sigma}\to(0,\infty)\), define the suspension manifold
\[
M
:=
\bigl\{(q,s):q\in\widehat{\Sigma},\ 0\le s\le \rho(q)\bigr\}/\!\sim ,
\]
where $\sim$ is the equivalence relation generated by $(q,\rho(q))\sim(Bq,0)$.
Equivalently,
\begin{equation}\label{eq:def of group}
M=
(\widehat{\Sigma}\times\mathbb R)/\langle F \rangle,
\quad
F(q,s):=(Bq,s-\rho(q)),
\end{equation}
where $\langle F\rangle=\{F^n:n\in\mathbb Z\}.$

\begin{proposition}\label{prop:m is compact}
The quotient $M$ is a compact connected smooth three-dimensional manifold without boundary.
\end{proposition}

\begin{proof}
Since $\widehat{\Sigma}$ is compact and $\rho>0$ is smooth,
\[
\mathcal F
=
\{(q,s):q\in\widehat{\Sigma},\ 0\le s\le\rho(q)\}
\]
is a compact connected smooth three-dimensional manifold with two boundary components,
\[
\partial_-\mathcal F=\{(q,0)\},
\qquad
\partial_+\mathcal F=\{(q,\rho(q))\}.
\]
The identification $(q,\rho(q))\longmapsto(Bq,0)$ is a diffeomorphism from $\partial_+\mathcal F$ onto $\partial_-\mathcal F$. Gluing these two boundary components therefore gives a compact smooth three-dimensional manifold $M$ without boundary. Since the quotient map $\mathcal F\to M$ is continuous and surjective, $M$ is connected.
\end{proof}

Vertical translation
\[
T_t(q,s):=(q,s+t),\quad t\in\mathbb R,
\]
commutes with $F$ and hence descends to a smooth flow on $M$, given by
\[
\phi^t([q,s]):=[q,s+t],\quad t\in\mathbb R.
\]
Its generating vector field is $X_0=\partial_s.$ Since $B:\widehat{\Sigma}\to\widehat{\Sigma}$
is an Anosov diffeomorphism, as shown in \eqref{eq:seed-eigenvalues}, and $\rho$ is positive and smooth, the corresponding suspension flow is Anosov; see \cite[Chapter~1, Exercise~1.2]{BP13}. 

\begin{proposition}\label{prop:seed-volume}
The suspension flow \(\phi^t\) preserves the smooth volume form $\Omega=dq_1\wedge dq_2\wedge ds.$ Its total volume and normalized probability are
\begin{equation}\label{eq:seed-volume-normalization}
Z=\int_{\widehat\Sigma}\rho(q)\,dq_1dq_2 <\infty,
\quad
 \mu=Z^{-1}\Omega.
\end{equation}
\end{proposition}

\begin{proof}
We first verify that $\Omega$ is well-defined on the quotient $M$. On $\widehat{\Sigma}\times\mathbb R$, let
$\widetilde{\Omega}
:=
dq_1\wedge dq_2\wedge ds.$ Since $B:\widehat{\Sigma}\to\widehat{\Sigma}$ preserves area,
\[
\begin{aligned}
F^*\widetilde{\Omega}
&=
B^*(dq_1\wedge dq_2)\wedge d(s-\rho)=
dq_1\wedge dq_2\wedge(ds-d\rho)\\
&=
dq_1\wedge dq_2\wedge ds=
\widetilde{\Omega},
\end{aligned}
\]
where $dq_1\wedge dq_2\wedge d\rho=0$ due to $\rho=\rho(q_1,q_2)$. Hence $\widetilde{\Omega}$ descends to a smooth volume form $\Omega$ on $M$.

The suspension vector field is $X_0=\partial_s.$ Therefore, $\operatorname{div}_{\Omega}X_0
=0.$ Thus the suspension flow $\phi^t$ preserves $\Omega$. Finally, 
\[
Z:=
\int_M\Omega=
\int_{\widehat{\Sigma}}
\int_0^{\rho(q)}
ds\,dq_1\,dq_2=
\int_{\widehat{\Sigma}}
\rho(q)\,dq_1\,dq_2 < \infty,
\]
where the last inequality follows from the boundedness of \(\rho\) and the compactness of \(\widehat\Sigma\). Hence $\mu:=Z^{-1}\Omega$ is an invariant probability measure for $\phi^t$.
\end{proof}

\subsection{Topological mixing}

We next prove that the suspension flow is topologically mixing. We use the following periodic-orbit criterion for hyperbolic flows.
\begin{proposition}\cite[Theorem~1.1(2)]{FMT}
\label{prop:period-criterion}
Let \(\varphi^t\) be a \(C^r (r \geq 1)\) Axiom~A flow satisfying the no-cycle condition, and let
\(L\) be a nontrivial basic set of \(\varphi^t\). Then \(\varphi^t|_L\) is not topologically mixing if and only if there exists \(c>0\) such that every periodic-orbit period in $L$ belongs to $c\mathbb Z.$
\end{proposition}

\begin{proposition}
\label{prop:seed-topological-mixing}
The suspension flow \(\phi^t\) is topologically mixing.
\end{proposition}

\begin{proof}

We first verify that $\phi^t$ and $M$ meet the conditions in Proposition~\ref{prop:period-criterion}.
Since \(B\) is a hyperbolic toral automorphism, it is topologically mixing. Moreover, its periodic points are dense in \(\widehat{\Sigma}\); see \cite[Exercise~1.7.1]{BS02}. If \(q\) is periodic for \(B\) with period \(n\), then its suspension orbit is periodic with period $T(q)=\sum_{j=0}^{n-1}\rho(B^j q).$
Since the periodic points of \(B\) are dense in \(\widehat{\Sigma}\),
the periodic orbits of \(\phi^t\) are dense in \(M\). Since $\phi^t$ is also a  smooth Anosov flow, it is a smooth Axiom A flow. 

Since $B$ is topologically mixing, there exists \(q\in\widehat{\Sigma}\) whose forward orbit under \(B\) is dense. We claim that $[q,0] \in M$ has a dense forward orbit under $\phi^t$. Thus, $\phi^t$ is topologically transitive. The spectral decomposition therefore consists of the single basic set \(M\), so the no-cycle condition holds.

Since \(M\) is the whole phase space, it is closed and locally maximal.
Thus \(M\) is a nontrivial hyperbolic basic set in the sense of \cite{FMT}.

Since the points in \(P\) and \(q_*\) are fixed by \(B\), their
suspension orbits close after a single roof crossing. Hence, by
\eqref{eq:seed-roof-values}, their least periods are \(1\) and \(\sqrt2\), respectively. Thus the set of periodic-orbit periods is not contained in any lattice \(c\mathbb Z\). By Proposition~\ref{prop:period-criterion}, \(\phi^t\) is topologically mixing.
\end{proof}

\subsection{Exponential mixing of the seed flow}

We now apply the exponential mixing theorem of \cite{TZ}.
The version needed here is the following.

\begin{theorem}\cite[Theorem~1.1]{TZ}
\label{thm:tz}
Let $\varphi^t$ be a topologically mixing $C^\infty$ Anosov flow on a compact three-dimensional manifold. Then, $\varphi^t$ is exponentially mixing with respect to any equilibrium measure with H\"older  potential and all H\"older test functions. 
More precisely, for each fixed $0<\alpha<1$, there exists
$\gamma_\alpha>0$ such that for every
$h,l\in C^\alpha(M)$, there exists $C_{h,l} >0$ such that
\[
|\Corr_{\eta,\varphi}(h,l;t)|
\le C_{h,l}e^{-\gamma_\alpha t},
\qquad t\ge0,
\]
where $\eta$ is the corresponding equilibrium measure.
\end{theorem}

By the preceding sections, $M$ is a compact smooth three-dimensional manifold and $\phi^t$ is a topologically mixing $C^\infty$ Anosov flow. Thus, in order to apply Theorem~\ref{thm:tz} to the invariant probability $\mu$, it remains only to identify $\mu$ as the equilibrium measure of a H\"older potential.

Let $E^u$ denote the unstable bundle of $\phi^t$, and let $J_t^u(x)$ denote the Jacobian of $D\phi_x^t:E_x^u\longrightarrow E_{\phi^t x}^u$, with respect to a fixed smooth Riemannian metric. Define the unstable geometric potential by
\[
F^u(x)
:=
-\left.\frac{d}{dt}\right|_{t=0}
\log J_t^u(x).
\label{eq:unstable-potential}
\]
An equilibrium measure for a potential $F$ is an invariant probability maximizing
$h_\eta(\phi^1)+\int F\,d\eta$ among all invariant probability measures $\eta$, where \(h_\eta(\phi^1)\) denotes the measure entropy of the time-one map \(\phi^1\) with respect to \(\eta\).

\begin{proposition}\cite{BR}
\label{prop:bowen-ruelle}
Let \(A\) be a basic hyperbolic set for a \(C^2\) flow. The unstable geometric potential \(F^u\) is H\"older continuous and has a unique equilibrium measure \(\mu_{F^u}\). Moreover, if $A=M$, then every invariant probability measure absolutely continuous with respect to smooth volume coincides with \(\mu_{F^u}\).
\end{proposition}

\begin{proof}
The H\"older regularity of \(F^u\) is
\cite[Lemma~4.1]{BR}, and uniqueness of its equilibrium measure follows
from \cite[Theorem~3.3]{BR}. In the case \(A=M\), \cite[Theorem~5.1 and Remark~5.2]{BR} show that every absolutely continuous invariant probability coincides with this equilibrium measure.
\end{proof}

By the proof of Proposition~\ref{prop:seed-topological-mixing}, the basic set considered here is the whole manifold \(A=M\). Since $\mu$ is a smooth invariant probability measure, it is absolutely continuous with respect to smooth volume. Proposition~\ref{prop:bowen-ruelle} therefore implies that $\mu$ is the equilibrium measure of the
H\"older potential $F^u$. Hence Theorem~\ref{thm:tz} applies to  $(\phi^t,\mu)$.
For every $0<\alpha<1$, there exists $\gamma_\alpha>0$ such that,
for every fixed $h,l\in C^\alpha(M)$, there exists $C_{h,l} >0$ such that
\begin{equation}\label{eq:pointwise-exp-mixing}
|\Corr_{\mu,\phi}(h,l;t)|
\le
C_{h,l}e^{-\gamma_\alpha t},
\quad t\ge0.
\end{equation}

For the transfer argument used later, we need a uniform estimate in
terms of the H\"older norms of the observables, with a constant
independent of the particular pair \(h,l\). This follows from the
uniform boundedness principle.

\begin{lemma}
\label{lem:uniform-bilinear}
Let $E$ be a Banach space and let $(B_t)_{t\ge0}$ be a family of
continuous bilinear forms on $E\times E$. If for every fixed $h,l\in E$,
\[
\sup_{t\ge0}|B_t(h,l)|<\infty,
\]
then $\sup_{t\ge0}\|B_t\|<\infty.$
\end{lemma}

\begin{proof}
For fixed $h\in E$, the family of bounded linear functionals
\[
l \longmapsto B_t(h,l)
\]
is pointwise bounded. The Banach–Steinhaus theorem gives
\[
\sup_{t\ge0}\|B_t(h,\cdot)\|_{E^*}<\infty.
\]
Now define
\[
T_t:E\to E^*,
\qquad
T_th:=B_t(h,\cdot).
\]
The family $(T_t)_{t\ge0}$ is pointwise bounded, so a second application of the Banach–Steinhaus theorem gives
\[
\sup_{t\ge0}\|T_t\|<\infty.
\]
Since $\|T_t\|=\|B_t\|,$ the result follows.
\end{proof}

\begin{proposition}
\label{prop:seed-mixing}
For every $0<\alpha<1$, there exist constants
$C_\alpha,\gamma_\alpha>0$ such that for all $h,l\in C^\alpha(M)$,
\begin{equation}\label{eq:seed-mixing}
|\Corr_{\mu,\phi}(h,l;t)|
\le
C_\alpha e^{-\gamma_\alpha t}
\|h\|_{C^\alpha(M)}
\|l\|_{C^\alpha(M)}, \quad t\ge0.
\end{equation}
\end{proposition}

\begin{proof}
Fix $0<\alpha<1$ and let $\gamma_\alpha>0$ be given by
Theorem~\ref{thm:tz}. Define
\[
B_t(h,l)
:=
e^{\gamma_\alpha t}
\Corr_{\mu,\phi}(h,l;t),
\quad
h,l\in C^\alpha(M).
\]
For each fixed $t$, $B_t$ is a continuous bilinear form on
$C^\alpha(M)\times C^\alpha(M)$. Moreover,
\eqref{eq:pointwise-exp-mixing} gives
\[
\sup_{t\ge0}|B_t(h,l)|\leq C_{h,l} <\infty
\]
for every fixed pair $h,l$.
Lemma~\ref{lem:uniform-bilinear} therefore yields $\sup_{t\ge0}\|B_t\|<\infty.$
Hence there exists $C_\alpha>0$ such that
\[
|B_t(h,l)|
\le
C_\alpha
\|h\|_{C^\alpha(M)}
\|l\|_{C^\alpha(M)}.
\]
Then the definition of \(B_t\) gives \eqref{eq:seed-mixing}.
\end{proof}

\section{Exact solid-torus coordinates near the four periodic orbits}
\label{sec:charts}

We introduce local coordinates near the four selected periodic orbits for the surgery in the next section. Let $
\widehat{\pi}:\mathbb{R}^2 \rightarrow\widehat{\Sigma}$ 
denote the quotient map. Let $\lambda:=\log \Lambda$, where $\Lambda$ is the eigenvalue of $B$ defined in \eqref{eq:seed-eigenvalues}. Choose \(Q\in SO(2)\) such that
\begin{equation}\label{eq:charts-diagonalization}
 Q^{-1}BQ=\operatorname{diag}(\Lambda,\Lambda^{-1}).
\end{equation}
For \(p\in P\), let
\[
\Gamma_p:=\{[p,s]:s\in\mathbb{R}/\mathbb{Z}\}
\]
denote the corresponding periodic orbit of the suspension flow. Set $\Gamma:=\bigcup_{p\in P}\Gamma_p.$

For each \(p\in P\), choose a lift
$\widetilde p\in\mathbb{R}^2,\widehat{\pi}(\widetilde p)=p.$
Since \(P\) is finite, we may choose \(r_0>0\), independent of \(p\),
such that
\begin{equation}\label{eq:def of kp}
\kappa_p:D_{r_0} \longrightarrow\widehat{\Sigma},
\qquad
\kappa_p(u,v)
:=
\widehat{\pi}\!\left(
\widetilde p+
Q
\begin{pmatrix}
u\\ v
\end{pmatrix}
\right),
\end{equation}
is a smooth coordinate chart centered at \(p\). Here $D_{r_0}=\{(u,v)\in \mathbb{R}^2: u^2+v^2 < r_0^2 \}$. The images
\(\kappa_p(D_{r_0})\) are pairwise disjoint, and the roof function is
identically equal to \(1\) on these neighborhoods. Since \(Bp=p\) on \(\widehat\Sigma\) and \eqref{eq:charts-diagonalization} holds, we have
\[
B\kappa_p(u,v)
=
\widehat{\pi}\!\left(
B\widetilde p+
BQ
\begin{pmatrix}
u\\ v
\end{pmatrix}
\right)                                    =
\widehat{\pi}\!\left(
\widetilde p+
Q
\begin{pmatrix}
\Lambda u\\
\Lambda^{-1}v
\end{pmatrix}
\right)                                    =
\kappa_p(\Lambda u,\Lambda^{-1}v),
\]
whenever $(u,v),(\Lambda u,\Lambda^{-1}v)\in D_{r_0}.$
Thus, near the periodic orbit \(\Gamma_p\), the suspension
identification takes the form
\[
(u,v,1)\sim(\Lambda u,\Lambda^{-1}v,0).
\]

\begin{proposition}
\label{prop:solid-torus-coordinates}
There exists \(R>0\) such that, for every \(p\in P\), there is a
\(C^\infty\)-diffeomorphism
\[
\Psi_p:D_R\times\mathbb S^1\longrightarrow U_p,
\]
where \(U_p\) is an open neighborhood of \(\Gamma_p\). The four neighborhoods \(U_p\) may be chosen pairwise disjoint.
For \(s\in[0,1]\) representing \([s]\in\mathbb S^1\), the map is given by
\begin{equation}
\label{eq:solid-torus-chart}
\Psi_p(x,y,[s])
=
\left[
\kappa_p
\left(
e^{-\lambda s}x,
e^{\lambda s}y
\right),
s
\right].
\end{equation}
In the resulting coordinates \((x,y,s)\) on \(U_p\),
\begin{equation}
\label{eq:solid-torus-normal-form}
\Gamma_p=\{x=y=0\},
\qquad
X_0
=
\lambda x\,\partial_x
-
\lambda y\,\partial_y
+
\partial_s,
\qquad
\Omega
=
dx\wedge dy\wedge ds.
\end{equation}
\end{proposition}

\begin{proof}
Fix a small \(\varepsilon>0\).  Choose \(R>0\) so that $e^{\lambda(1+\varepsilon)}R<r_0.$ Then, for every \((x,y)\in D_R\) and
\(s\in(-\varepsilon,1+\varepsilon)\), $(
e^{-\lambda s}x,
e^{\lambda s}y )\in D_{r_0}$.
So the expression in \eqref{eq:solid-torus-chart} is defined in the chosen base chart.

We first verify the compatibility across the suspension seam. By direct computation, we have
\[
\begin{aligned}
&B\,
\kappa_p
\left(
e^{-\lambda(s+1)}x,
e^{\lambda(s+1)}y
\right)                                    =
\kappa_p
\left(
\Lambda e^{-\lambda(s+1)}x,
\Lambda^{-1}e^{\lambda(s+1)}y
\right)                                    =
\kappa_p
\left(
e^{-\lambda s}x,
e^{\lambda s}y
\right).
\end{aligned}
\]
Since the suspension relation is
$(q,s+1)\sim(Bq,s)$ for $q$ near $p$,
we obtain
\[
\left[
\kappa_p
\left(
e^{-\lambda(s+1)}x,
e^{\lambda(s+1)}y
\right),
s+1
\right]
=
\left[
\kappa_p
\left(
e^{-\lambda s}x,
e^{\lambda s}y
\right),
s
\right].
\]
Thus the two local expressions agree on an open neighborhood of the suspension seam. Therefore \eqref{eq:solid-torus-chart} descends from \(D_R\times[0,1]\) to a well-defined \(C^\infty\) map
\[
\Psi_p:D_R\times\mathbb{S}^1\longrightarrow M.
\]

The inverse coordinates follow directly from
$u=e^{-\lambda s}x, v=e^{\lambda s}y$. Since 
\[
e^{\lambda(s-1)}(\Lambda u)=e^{\lambda s}u,
\qquad
e^{-\lambda(s-1)}(\Lambda^{-1}v)=e^{-\lambda s}v,
\]
these expressions are compatible with the suspension identification. Hence, for a representative $0 \leq s \leq 1$,
\[
[\kappa_p(u,v),s]
\longmapsto
\bigl(e^{\lambda s}u,e^{-\lambda s}v,[s]\bigr)
\]
defines a smooth inverse to \(\Psi_p\) on its image. 

To show that the image is open, consider the change of variables in the lifted coordinates. Since \(\kappa_p\) is a local \(C^\infty\)-diffeomorphism and
\[
(x,y,s)
\longmapsto
(u,v,s)
=
\bigl(e^{-\lambda s}x,e^{\lambda s}y,s\bigr)
\]
has Jacobian determinant
\[
\det
\frac{\partial(u,v,s)}
     {\partial(x,y,s)}
=
1,
\]
\(\Psi_p\) is a local \(C^\infty\)-diffeomorphism. In particular, it is an open map.

Define $U_p:=\Psi_p(D_R\times\mathbb S^1)$. Then $U_p$ is open. Since $\Psi_p(0,0,[s])=[p,s]$, $U_p$ is an open neighborhood of $\Gamma_p$. The map
\[
\Psi_p:D_R\times\mathbb S^1\longrightarrow U_p
\]
is a \(C^\infty\)-diffeomorphism.  By decreasing \(R\), the four neighborhoods \(U_p\) may be chosen pairwise disjoint.

It remains to compute the local expressions in \eqref{eq:solid-torus-normal-form}. The periodic orbit $\Gamma_p$ corresponds to $u=v=0,$ and therefore to $x=y=0.$

In the local $(u,v,s)$ suspension coordinates, the flow is vertical, so $X_0=\partial_s$. Using $x=e^{\lambda s}u, y=e^{-\lambda  s}v,$ we obtain
\[
\dot x
=
\lambda e^{\lambda s}u
=
\lambda x,
\quad
\dot y
=
-\lambda e^{-\lambda s}v
=
-\lambda y.
\]
Therefore
\[
X_0
=
\lambda x\,\partial_x
-
\lambda y\,\partial_y
+
\partial_s.
\]

Finally, in the local $(u,v,s)$-coordinates, since \(Q\in SO(2)\),  $\Omega=dq_1 \wedge dq_2 \wedge ds= du\wedge dv\wedge ds.$
From $u=e^{-\lambda s}x, v=e^{\lambda s}y,$ we have
\[
du
=
e^{-\lambda s}(dx-\lambda x\,ds),
\quad
dv
=
e^{\lambda s}(dy+\lambda y\,ds).
\]
Therefore
\[
\begin{aligned}
\Omega
&=
du\wedge dv\wedge ds=
(dx-\lambda x\,ds)
\wedge
(dy+\lambda y\,ds)
\wedge ds=
dx\wedge dy\wedge ds.
\end{aligned}
\]
This proves \eqref{eq:solid-torus-normal-form}.
\end{proof}

\section{Smooth volume-preserving surgery}
\label{sec:surgery}

We now modify $M$ near the four selected periodic orbits by removing a solid-torus neighborhood of each orbit and filling it back in with a different boundary identification. The resulting manifold will be denoted by $N$. In Section~\ref{sec:topology} we will prove that $N\cong \mathbb T^3.$
The purpose of this section is to carry out this surgery while keeping control of the dynamics. 

\subsection{The coordinate exchange}
\label{subsec:the coordinate exchange}

We now describe the new filling near one of the selected periodic orbits. Recall from Proposition~\ref{prop:solid-torus-coordinates} that a neighborhood $U_p$ of such an orbit $\Gamma_p$ has coordinates $(x,y,s)\in D_R\times \mathbb S^1,$ and the periodic orbit itself is given by $x=y=0$.

On the punctured disk $x^2+y^2>0$, write
$x=r\cos\theta,y=r\sin\theta,$
where $r>0$ and $\theta\in\mathbb R/(2\pi\mathbb Z)$. We introduce new coordinates by
\begin{equation}
\label{eq:coordinate-exchange}
\xi=r\cos(2\pi s),\quad
\eta=r\sin(2\pi s),\quad
z=-\theta.
\end{equation}
Thus the old orbit coordinate $s$ becomes the angular coordinate in the new disk, while the old disk angle $\theta$ becomes the new circle coordinate $z$.

For $r>0$, this change of coordinates is invertible:
\begin{equation}\label{eq:inverse formula 1}
r=(\xi^2+\eta^2)^{1/2},\quad
s=\frac{\arg(\xi+i\eta)}{2\pi}\pmod 1,
\end{equation}
and
\begin{equation}\label{eq:inverse formula 2}
    x=r\cos z, y=-r\sin z.
\end{equation}
Hence \eqref{eq:coordinate-exchange} is a diffeomorphism away from the cores.
Fix $0<r_1<R$. Let
\[
\mathcal M_p=\{q\in U_p:r(q)\le r_1\}
\]
be the smaller tube around $\Gamma_p$, and introduce a new tube
\[
\widetilde{\mathcal M}_p
=
\left\{ (\xi,\eta,z):
\xi^2+\eta^2\le r_1^2,\quad
z\in\mathbb R/(2\pi\mathbb Z)
\right\}.
\]
On the boundary $r=r_1$, \eqref{eq:coordinate-exchange} gives
\begin{equation}\label{eq:def of gp}
g_p:\partial\mathcal M_p
\longrightarrow
\partial\widetilde{\mathcal M}_p,
\qquad
(r_1 \cos \theta, r_1 \sin \theta,s)
\longmapsto
\bigl(r_1\cos(2\pi s),r_1\sin(2\pi s),-\theta\bigr).
\end{equation}
This map is well-defined because $s$ is taken modulo $1$ and $\theta,z$ modulo $2\pi$. The inverse formula above shows that $g_p$ is a smooth diffeomorphism.

We remove the interiors of the four tubes $\mathcal M_p$ and attach the new tubes $\widetilde{\mathcal M}_p$ using the maps $g_p$. Denote the resulting space by
\begin{equation}\label{eq:def of N}
N
=
\left(
M\setminus
\bigcup_{p\in P}\operatorname{int}\mathcal M_p
\right)
\cup_{g_p}
\bigcup_{p\in P}\widetilde{\mathcal M}_p.
\end{equation}
We choose the smooth structure near each gluing torus so that the coordinates \((\xi,\eta,z)\) in \eqref{eq:coordinate-exchange} extend across \(r=r_1\) to the retained part of \(U_p\). 

Recall that $\Gamma=\bigcup_{p\in P}\Gamma_p$ is the union of the four old core orbits. Let $K\subset N$ be the union of the four new cores $\{\xi=\eta=0\}$.
We define
\[
\Phi:M\setminus\Gamma\longrightarrow N\setminus K
\]
by \eqref{eq:coordinate-exchange} when $0<r<r_1$, and by the natural
identification on the unchanged part of $M$. On $r=r_1$ the two definitions agree by the definition of the gluing map $g_p$.

\begin{proposition}
\label{prop:complement-map}
The space $N$ is a compact connected smooth three-dimensional manifold without boundary, and $\Phi$ is a diffeomorphism.
\end{proposition}

\begin{proof}
Since each $g_p$ is a smooth diffeomorphism, the standard smooth gluing construction gives $N$ the structure of a smooth three-dimensional manifold without
boundary; see \cite[Chapter~8, \S2]{Hirsch1976}. The set $M\setminus\bigcup_{p\in P}\operatorname{int}\mathcal M_p$
is compact, and each new tube $\widetilde{\mathcal M}_p$ is compact.
Their finite disjoint union is therefore compact. Since $N$ is its
quotient under the boundary identifications $g_p$ by \eqref{eq:def of N}, $N$ is compact.

The map $\Phi$ is given by \eqref{eq:coordinate-exchange} in each tube and by the natural identification outside the tubes. By the choice of the smooth gluing structure above, these two definitions fit together smoothly across each boundary $r=r_1$. The inverse formulas \eqref{eq:inverse formula 1} and \eqref{eq:inverse formula 2} give the same conclusion for $\Phi^{-1}$. Hence $\Phi$ is a diffeomorphism.

Finally, since a path in the
three-dimensional manifold $M$ can be perturbed to avoid the finite
union of embedded circles $\Gamma$, $M\setminus\Gamma$ is connected. Hence $N\setminus K$, and therefore $N$, is connected.
\end{proof}

\subsection{The new vector field and volume}

Since $\Phi$ is a diffeomorphism, it allows us to transfer both the vector field and the volume form to
$N\setminus K$. Define
\begin{equation}\label{eq:def of Y}
Y=\Phi_*X_0,
\quad
Y_{\Phi(q)}=D\Phi_q(X_0(q)),
\quad q \in M\setminus\Gamma,
\end{equation}
and
\begin{equation}\label{eq:def of volume}
    \Omega_N:=(\Phi^{-1})^*\Omega \quad\text{on }N\setminus K.
\end{equation}
We now compute these two objects in the new coordinates. The main point is that both formulas remain smooth at $\xi=\eta=0$, and hence extend across the new cores $K$.

\begin{proposition}
\label{prop:new-vector-field}
In each new tube, 
\begin{equation}
\label{eq:new-vector-field}
Y
=
\bigl(\lambda\cos(2z)\xi-2\pi\eta\bigr)\partial_\xi
+
\bigl(2\pi\xi+\lambda\cos(2z)\eta\bigr)\partial_\eta
-
\lambda\sin(2z)\partial_z.
\end{equation}
This formula extends $Y$ smoothly across the new core $\{\xi=\eta=0\}$. Hence $Y$ extends to a smooth vector field on all of $N$. Its flow $\Psi^t$ is defined for all $t\in\mathbb R$.
\end{proposition}

\begin{proof}
By \eqref{eq:solid-torus-normal-form}, \(\dot x=\lambda x\), \(\dot y=-\lambda y\), and \(\dot s=1\). Hence, for $r>0$,
\[
\dot r
=
\frac{x\dot x+y\dot y}{r}
=
\lambda r\cos(2\theta),
\qquad
\dot\theta
=
\frac{x\dot y-y\dot x}{r^2}
=
-\lambda\sin(2\theta).
\]
Since $z=-\theta$,
\[
\dot z=-\lambda\sin(2z).
\]
Using $\xi=r\cos(2\pi s)$ and $
\eta=r\sin(2\pi s),$ we obtain
\[
\dot\xi
=
\dot r\cos(2\pi s)
-
2\pi r\sin(2\pi s)\dot s =
\lambda\cos(2z)\xi-2\pi\eta,
\]
and
\[
\dot\eta
=
\dot r\sin(2\pi s)
+
2\pi r\cos(2\pi s)\dot s =
2\pi\xi+\lambda\cos(2z)\eta.
\]
This gives \eqref{eq:new-vector-field} on the new tube away from the core.

The coefficients in \eqref{eq:new-vector-field} are smooth functions of $(\xi,\eta,z)$ and remain well-defined when $\xi=\eta=0$.
Therefore the same formula defines a smooth extension of $Y$ across the new core.

Doing this in each of the four new tubes gives a smooth vector field on all of $N$. Since $N$ is compact, every smooth vector field on $N$ is complete. Hence its flow $\Psi^t$ is defined for all
$t\in\mathbb R$.
\end{proof}

The same phenomenon occurs for the volume form.
\begin{proposition}
\label{prop:new-volume}
In each new tube,
\begin{equation}\label{eq:new volume form}
\Omega_N
=
\frac{1}{2\pi}\,d\xi\wedge d\eta\wedge dz.
\end{equation}
This formula extends $\Omega_N$ smoothly across the new core $\{\xi=\eta=0\}$. Hence $\Omega_N$ extends to a smooth positive volume form on all of $N$. The extended volume form is invariant under \(\Psi^t\) and has total mass \(Z\) as defined in \eqref{eq:seed-volume-normalization}. Consequently, \(\nu:=Z^{-1}\Omega_N\) is a smooth invariant probability measure. Moreover, \(\nu=\Phi_*\mu\). 
\end{proposition}

\begin{proof}
On the punctured old tube, by \eqref{eq:solid-torus-normal-form}
\[
\Omega
=
dx\wedge dy\wedge ds
=
r\,dr\wedge d\theta\wedge ds.
\]
From \eqref{eq:coordinate-exchange},
\[
d\xi\wedge d\eta
=
2\pi r\,dr\wedge ds,
\qquad
dz=-d\theta.
\]
Hence
\[
\frac1{2\pi}
d\xi\wedge d\eta\wedge dz
=
r\,dr\wedge d\theta\wedge ds
=
\Omega.
\]
Therefore \eqref{eq:new volume form} is exactly the pushed-forward volume form on the punctured tube. Since its coefficients are constant, it extends smoothly and remains nonzero at $\xi=\eta=0$.
Hence, by \eqref{eq:new-vector-field} and direct computation, 
\[
\operatorname{div}_{\Omega_N} Y
=
\lambda\cos(2z)+\lambda\cos(2z)-2\lambda\cos(2z)
=0
\]
in every new tube. Outside the tubes $\operatorname{div}_{\Omega_N} Y= \operatorname{div}_{\Omega} X_0=0$ by $Y=\Phi_* X_0$. Both $\Gamma$ and $K$ are finite unions of smooth curves and therefore have zero volume. Hence
\[
\int_N\Omega_N
=
\int_{N\setminus K}\Omega_N
=
\int_{M\setminus\Gamma}\Omega
=
\int_M\Omega
=
Z.
\]
Hence, $\nu:=Z^{-1}\Omega_N$ defines an invariant probability measure on $N$ for $\Psi^t$. The change of variables gives
$\nu=\Phi_*\mu$.
\end{proof}

\subsection{The equilibria on the new cores}

The vector field $Y$ is now defined smoothly on all of $N$.
We next determine its equilibria and verify that they are hyperbolic.

\begin{proposition}
\label{prop:new-equilibria}
$Y$ has exactly $16$ hyperbolic equilibria. The zero set of $Y$ is
\[
\{e_{p,k}:p\in P,\ k=0,1,2,3\},
\quad
e_{p,k}=(0,0,k\pi/2),
\]
where \(e_{p,k}\) is expressed in the coordinates of the \(p\)-th new tube.
\end{proposition}

\begin{proof}
Let $q=(\xi,\eta,z)$ be an equilibrium of $Y$ in one of the new tubes. By \eqref{eq:new-vector-field}, the first two components give 
\[
\begin{pmatrix}
\lambda\cos(2z) & -2\pi\\
2\pi & \lambda\cos(2z)
\end{pmatrix}
\binom{\xi}{\eta}
=0.
\]
Since the determinant is $\lambda^2\cos^2(2z)+4\pi^2>0,$
the matrix is invertible. Hence, $\xi=\eta=0.$
The third component then gives $\sin(2z)=0$, so $
z=0,\frac{\pi}{2},\pi,\frac{3\pi}{2}\pmod{2\pi}.$ Outside the new tubes, $Y=\Phi_*X_0$, and $X_0$ has no zeros.
Hence these are all equilibria of $Y$.
By direct computation, the eigenvalues of $DY(e_{p,k})$ are $(-1)^k\lambda\pm2\pi i$ and $2 (-1)^{k+1}\lambda.$ 
Since $\lambda>0$, every equilibrium is hyperbolic.
\end{proof}

\subsection{Conjugacy with the same elapsed time}

The union \(\Gamma\) of the old core orbits is invariant under $\phi^t$. On each new core $K_p=\{\xi=\eta=0\}$, the vector field reduces to $Y=-\lambda\sin(2z)\partial_z$, so $K$ is invariant under $\Psi^t$. By uniqueness of solutions, their complements $M\setminus\Gamma$ and $N\setminus K$ are invariant as well. We now show that $\Phi$ sends each old trajectory to the corresponding new trajectory without changing the time parameter.

\begin{proposition}
\label{prop:same-time-conjugacy}
For every $x\in M\setminus\Gamma$ and $t\in\mathbb R$,
\begin{equation}\label{eq:two flows time relation}
\Phi(\phi^t x)=\Psi^t(\Phi(x)).
\end{equation}
Thus the two flows are conjugate on the complements of the cores with the same elapsed time.
\end{proposition}

\begin{proof}
Fix $x\in M\setminus\Gamma$ and set
$\gamma(t)=\Phi(\phi^t x).$
Since $M\setminus\Gamma$ is invariant, $\gamma(t)$ is defined for all
$t\in\mathbb R$. Using $Y=\Phi_*X_0$,
\[
\dot\gamma(t)
=
D\Phi_{\phi^t x}\bigl(X_0(\phi^t x)\bigr)
=
Y(\gamma(t)).
\]
Moreover, $\gamma(0)=\Phi(x)$. By uniqueness of solutions of the ordinary differential equation generated by $Y$,
$\gamma(t)=\Psi^t(\Phi(x)).$
Hence \eqref{eq:two flows time relation} holds.
\end{proof}

\section{Identification of the surgered manifold}\label{sec:topology}

The local construction determines a smooth manifold $N$, but does not yet
identify its topology. We do this through the free action of $(\Z/2\Z)^2$. After removing the variable roof from the description of $M$, the quotient surgery becomes a familiar filling of a punctured-torus mapping torus. Its result is the Hantzsche--Wendt manifold $W$.
It remains to show that the particular four-sheeted cover $N\to W$ is the
translation cover, whose total space is $\T^3$. Notice that $W$ itself is
not a three-torus.

\subsection{The mapping torus and the filling slope}

For a torus $S$, a diffeomorphism $C:S\to S$, and a smooth positive function $h:S\to(0,\infty)$, write
\[
M_h(C;S)
=
\bigl\{(q,s):q\in S,\ 0\le s\le h(q)\bigr\}
\big/\bigl((q,h(q))\sim(Cq,0)\bigr),
\]
and denote its points by $[q,s]_h$. In particular,
\[
M_1(C;S)
=
\bigl(S\times[0,1]\bigr)
\big/\bigl((q,1)\sim(Cq,0)\bigr).
\]
In this notation,
\[
M=M_\rho(B;\widehat\Sigma).
\]
Set $\Sigma=\R^2/\Z^2$ and let $\pi:\widehat\Sigma\to\Sigma$ be the natural projection. Since $B\in\mathrm{SL}(2,\Z)$ and $\rho$ is $\Z^2$-periodic,
they induce a diffeomorphism and a smooth roof function on $\Sigma$, denoted by the same symbols. 
Thus
\[
\pi\circ B=B\circ\pi,
\qquad
\rho\circ\pi=\rho.
\]
The induced map $B:\Sigma\to\Sigma$ fixes $0$. 
By Lemma~\ref{lem:seed-roof}, the induced roof function $\rho$ satisfies $\rho\ge1$ on $\Sigma$ and $\rho\equiv1$ in a neighborhood of the origin.

\begin{proposition}\label{prop:top-roof}
	There is a smooth diffeomorphism
	\[
	F_\Sigma:M_1(B;\Sigma)\longrightarrow M_\rho(B;\Sigma)
	\]
	satisfying
	\[
	F_\Sigma([x,u]_1)=[x,u]_\rho
	\quad\text{whenever }\rho(x)=1.
	\]
	In particular, $F_\Sigma$ is the identity in the common
	suspension coordinates near the orbit over $0\in\Sigma$.
\end{proposition}

\begin{proof}
	Choose a smooth nondecreasing function $c:[0,1]\to[0,1]$ that is zero near zero and one near one. Set
	\begin{equation}\label{eq:top-height}
		H(x,u)=u+(\rho(x)-1)c(u).
	\end{equation}
	Then
	\[
	\partial_uH(x,u)
	=
	1+(\rho(x)-1)c'(u)\ge1,
	\qquad
	H(x,0)=0,
	\qquad
	H(x,1)=\rho(x).
	\]
	Thus $H(x,\cdot)$ is a bijection from $[0,1]$ onto $[0,\rho(x)]$, with inverse depending smoothly on $(x,s)$.
	Since
	\[
	[x,H(x,1)]_\rho=[Bx,H(Bx,0)]_\rho,
	\]
	the formula
	\[
	F_\Sigma([x,u]_1)=[x,H(x,u)]_\rho
	\]
	defines a bijection, smooth with smooth inverse away from the suspension seam.
	
	For sufficiently small $\varepsilon>0$, use the seam charts
	\[
	\chi_h(x,t)=
	\begin{cases}
		[x,t]_h,
		& 0\le t<\varepsilon,\\[2pt]
		[B^{-1}x,h(B^{-1}x)+t]_h,
		& -\varepsilon<t<0,
	\end{cases}
	\qquad h\in\{1,\rho\},
	\]
	where $x$ ranges over a small coordinate neighborhood in $\Sigma$.
	The endpoint conditions on $c$ give
	\[
	H(x,t)=t
	\quad (0\le t<\varepsilon),
	\qquad
	H(B^{-1}x,1+t)=\rho(B^{-1}x)+t
	\quad (-\varepsilon<t<0).
	\]
	Hence
	\[
	\chi_\rho^{-1}\circ F_\Sigma\circ\chi_1(x,t)=(x,t).
	\]
	Therefore $F_\Sigma$ and its inverse are smooth across the seam.
	
	Finally, $\rho(x)=1$ implies $H(x,u)=u$. Since $\rho=1$
	near $0$, this proves the asserted local identity.
\end{proof}

The map $F_\Sigma$ is used only to identify the underlying smooth manifolds; it does not, in general, conjugate their natural suspension flows with the same time parameter.
The comparison of elapsed times and correlations uses the conjugacy $\Phi$ from Proposition~\ref{prop:same-time-conjugacy}.

For $p\in P$, use the coordinates $(x,y,s)$ of Proposition~\ref{prop:solid-torus-coordinates} and write
$x=r\cos\theta$, $y=r\sin\theta$, as in
Section~\ref{subsec:the coordinate exchange}.
On $\partial \mathcal M_p=\{r=r_1\}$, orient the meridian $m_{\mathrm{old}}$ by increasing $\theta$ at fixed $s$,
and the longitude $\ell_{\mathrm{old}}$ by increasing $s$ at fixed $\theta$.
On $\partial\widetilde { \mathcal{M}}_p=\{\xi^2+\eta^2=r_1^2\}$,
orient the meridian $m_{\mathrm{new}}$ counterclockwise in the $(\xi,\eta)$-plane at fixed $z$, and the longitude
$\ell_{\mathrm{new}}$ by increasing $z$ at fixed $(\xi,\eta)$.

A slope is the unoriented isotopy class of an essential simple closed curve on a boundary torus. Replacing $\Gamma_p$ by its circle of transverse directions
gives a boundary torus, which we identify with $\partial \mathcal M_p$ using $(\theta,s)$. The \emph{degeneracy slope} is the slope represented by the
suspension of a stable or unstable eigenray on this torus.

\begin{proposition}\label{prop:top-slope}
	For every $p\in P$, the inverse of the gluing map $g_p:\partial \mathcal M_p\longrightarrow\partial\widetilde{\mathcal{M}}_p$ defined in \eqref{eq:def of gp} induces an isomorphism
    \[
    (g_p^{-1})_*:    H_1(\partial\widetilde{\mathcal{M}}_p;\Z)
    \longrightarrow H_1(\partial \mathcal M_p;\Z)
    \]
    satisfying
	\begin{equation}
		\label{eq:top-slope}
		(g_p^{-1})_*[m_{\mathrm{new}}]
		=[\ell_{\mathrm{old}}],
		\qquad
		(g_p^{-1})_*[\ell_{\mathrm{new}}]
		=-[m_{\mathrm{old}}].       
	\end{equation}
	Consequently, the filling is along the primitive degeneracy
	slope represented by $\ell_{\mathrm{old}}$, with no
	meridional twist. 
	
	Moreover, under the three-sheeted suspension covering
	\[
	C_3:M_1(B;\Sigma)\longrightarrow M_1(A;\Sigma),
	\qquad B=A^3,
	\]
	the framing near the orbit over $0$ defined by \eqref{eq:solid-torus-chart} is the pullback of the
	corresponding framing for $A$.
\end{proposition}

\begin{proof}
	On $\partial M_p$, the coordinate
	exchange~\eqref{eq:coordinate-exchange} is
	\[
	\xi=r_1\cos(2\pi s),\qquad
	\eta=r_1\sin(2\pi s),\qquad
	z=-\theta.
	\]
	A new meridian pulls back to a curve increasing $s$ by one at fixed $\theta$, whereas a new longitude pulls back to a curve decreasing $\theta$ by $2\pi$ at fixed $s$.
	This proves~\eqref{eq:top-slope}.
	
	By~\eqref{eq:solid-torus-normal-form}, the induced flow on the circle of transverse directions satisfies
	\[
	\dot\theta=-\lambda\sin(2\theta),
	\qquad
	\dot s=1.
	\]
	The stable and unstable eigenrays occur at
	$\theta=0,\pi/2,\pi,3\pi/2$.
	Each suspends to a closed curve $s\mapsto(\theta_0,s)$,
	representing $[\ell_{\mathrm{old}}]$ under the boundary identification above.
	This class is primitive because
	$[m_{\mathrm{old}}]$ and $[\ell_{\mathrm{old}}]$
	form a basis of $H_1(\partial M_p;\mathbb{Z})$.
	Since $B$ fixes $p$ and preserves each eigenray, the period and prong rotation in \cite[Convention~2.4]{Zung} are $m_i=1$ and
	$\omega_i=0$. Hence the degeneracy slope is the primitive class $[\ell_{\mathrm{old}}]$,
	corresponding to the zero slope $(0;1)$ in that convention.
	
	For the framing comparison, use the matrix $Q$ of
	\eqref{eq:charts-diagonalization}.
	Since $A$ is symmetric positive definite and $B=A^3$, $Q^{-1}AQ
	=\operatorname{diag}(e^{\lambda/3},e^{-\lambda/3}).$
	For $0\le s<1$, set $j=\lfloor3s\rfloor$ and $t=3s-j$.
	The suspension covering is
	\[
	C_3([q,s]_1)=[A^j q,t]_1.
	\]
	It is well defined by the suspension identifications and the relation $A^3=B$.
	In the common eigenbasis $Q$, the base representative in~\eqref{eq:solid-torus-chart} transforms as
	\[
	A^j Q
	\begin{pmatrix}
		e^{-\lambda s}x\\
		e^{\lambda s}y
	\end{pmatrix}
	=
	Q
	\begin{pmatrix}
		e^{-(\lambda/3)t}x\\
		e^{(\lambda/3)t}y
	\end{pmatrix}.
	\]
	Thus, in the corresponding solid-torus coordinates,
	\[
	C_3(x,y,s)=(x,y,3s\bmod1).
	\]
	In particular, the covering preserves the transverse coordinates $(x,y)$, proving the claimed compatibility
	of framings.
\end{proof}

\subsection{A regular four-sheeted quotient}

We now use the integer translations on the base to reduce the four surgeries on $M$ to a single filling on a quotient.
Let
\[
G=\mathbb Z^2/(2\mathbb Z)^2
\cong(\mathbb Z/2\mathbb Z)^2.
\]
For $v\in G$, consider the translation
\[
\widetilde\tau_v:
\widehat\Sigma\times\mathbb R
\longrightarrow
\widehat\Sigma\times\mathbb R,
\qquad
\widetilde\tau_v(q,s)=(q+v,s).
\]
\begin{proposition}
\label{prop:top-translations}
These translations descend to a free smooth action of $G$ on $M$, given by
\[
\tau_v([q,s]_\rho)=[q+v,s]_\rho.
\]
The action preserves $\Omega$ and commutes with $\phi^t$.
\end{proposition}

\begin{proof}
We first verify the translations commute with the identification map $F$ defined in \eqref{eq:def of group}.
Since $\rho$ is $\mathbb Z^2$-periodic and $(B-I)v\in(2\mathbb Z)^2$ for $v\in\mathbb G$, this commutation holds.
Thus the translations descend to smooth maps on $M$. The identities
\[
\tau_0=\mathrm{id}_M,
\qquad
\tau_v\circ\tau_w=\tau_{v+w}
\]
show that they define a smooth action of $G$.

If $\tau_v([q,s]_\rho)=[q,s]_\rho$, then there exists $n\in\mathbb Z$ such that 
$(q+v,s)=F^n(q,s)$.
Since $\rho>0$, equality of the second coordinates implies $n=0$. Hence $v=0$ in $G$, so the action is free.

The action preserves $\Omega$ because it is induced by translations in $q$, and it commutes with $\phi^t$ because it leaves $s$ unchanged.
\end{proof}

Let $r:M\to M/G$ denote the quotient map.

\begin{proposition}
	\label{prop:quotient-original}
	There is a natural diffeomorphism
	\[
	\mathcal Q:M/G\longrightarrow M_\rho(B;\Sigma),
	\quad
	\mathcal Q(r([q,s]_\rho))=[\pi(q),s]_\rho.
	\]
	Consequently,
	\begin{equation}\label{eq:quotient-unit-roof}
		M/G\cong M_1(B;\Sigma).
	\end{equation}
	Under these identifications, the four selected periodic orbits project to the suspension orbit over $0\in\Sigma$.
\end{proposition}

\begin{proof}

	The natural projection satisfies
	\[
	\pi^{-1}(\pi(q))=q+G,
	\qquad
	\pi(Bq)=B\pi(q),
	\qquad
	\rho(\pi(q))=\rho(q).
	\]
	Thus the map
	\[
	[q,s]_\rho\longmapsto[\pi(q),s]_\rho
	\]
	is invariant under the $G$-action and respects the suspension relation:
	\[
	[\pi(q),\rho(q)]_\rho
	=
	[B\pi(q),0]_\rho
	=
	[\pi(Bq),0]_\rho.
	\]
	It therefore induces $\mathcal Q$.
	
	Using representatives with $0\le s<\rho(q)$, the identity
	$\pi^{-1}(\pi(q))=q+G$ shows that $\mathcal Q$ is bijective.
	The induced map between the suspensions is a local
	diffeomorphism, since it is induced by
	$\pi\times\operatorname{id}$ in suspension charts.
	The quotient map $r$ is also a local diffeomorphism.
	Hence $\mathcal Q$ is a bijective local diffeomorphism, and therefore
	a diffeomorphism.

    By Proposition~\ref{prop:top-roof}, the map
    \[
    \Psi:=F_\Sigma^{-1}\circ \mathcal Q:
    M/G\longrightarrow M_1(B;\Sigma)
    \]
    is a diffeomorphism, proving~\eqref{eq:quotient-unit-roof}.	
	Finally,
	\[
	P=\pi^{-1}(0),
	\qquad
	\Gamma_p=\{[p,s]_\rho:s\in\R/\Z\},
	\qquad p\in P.
	\]
	Since $\rho(0)=1$ and $F_\Sigma([0,s]_1)=[0,s]_\rho$,
	\[
	\mathcal Q(r([p,s]_\rho))=[0,s]_\rho,
	\qquad
	\Psi(r([p,s]_\rho))=[0,s]_1.
	\]
	Thus all four selected periodic orbits have the asserted common image.
\end{proof}

We now recall some standard terminology from covering theory.

Let $p:E\to V$ be a covering map.
A \emph{deck transformation} of $p$ is a homeomorphism $d:E\to E$ satisfying $p\circ d=p$.
The covering $p$ is called \emph{regular} if its deck transformation group $\operatorname{Deck}(p)$ acts
transitively on each fiber $p^{-1}(v)$, $v\in V$. If $p$ is a smooth covering map, then every deck transformation of $p$ is a diffeomorphism. Throughout this paper, an unbranched covering is a covering map in the usual topological sense, including at the core circles.

We next show that the translation action extends across the surgery.

\begin{proposition}\label{prop:top-regular}
	The translation action on the retained part
	\[
	M\setminus\bigcup_{p\in P}\operatorname{int}\mathcal M_p
	\]
	extends to a free smooth action of $G$ on $N$, again denoted by $\tau_a$, $a\in G$.
	The manifold $N$ is connected, and the quotient map
	\begin{equation}\label{eq:top-regular-cover}
		p_N:N\longrightarrow W':=N/G
	\end{equation}
	is a smooth regular unbranched four-sheeted covering with deck transformation group $G$.
	Moreover, $W'$ is obtained from $M/G$ by the filling of
	Proposition~\ref{prop:top-slope} along the suspension orbit over $0\in\Sigma$.
\end{proposition}

\begin{proof}
For $a\in G$, the base charts in \eqref{eq:def of kp} satisfy
	\[
	\kappa_{p+a}(u,v)=\kappa_p(u,v)+a.
	\]
	The formula for $\Psi_p$ therefore gives
	\[
	\tau_a\circ\Psi_p=\Psi_{p+a},
	\qquad
	\tau_a(U_p)=U_{p+a},
	\qquad
	\tau_a(\mathcal M_p)=\mathcal M_{p+a}.
	\]
	Thus the retained part of $M$ is $G$-invariant, and translation preserves the tube coordinates $(x,y,s)$.

Extend the action to the new tubes by
\[
\tau_a:
\widetilde{\mathcal M}_p\longrightarrow \widetilde{\mathcal M}_{p+a},
\qquad
(\xi,\eta,z)_p\longmapsto(\xi,\eta,z)_{p+a}.
\]
Since the coordinate exchange \eqref{eq:coordinate-exchange} is independent of $p$, the gluing maps satisfy
\[
g_{p+a}\circ\tau_a=\tau_a\circ g_p \qquad\text{on }\partial\mathcal M_p.
\]
Therefore the action extends smoothly to $N$ and the group laws hold by construction.

The action is free on the retained part by Proposition~\ref{prop:top-translations}.
On the new tubes, every nonidentity element sends $\widetilde{\mathcal M}_p$ to a distinct tube $\widetilde{\mathcal M}_{p+a}$, so it fixes no point. $N$ is connected by Proposition~\ref{prop:complement-map}.
Fix $x_0\in N$. Since $G$ is finite and acts freely, there is an open neighborhood $U$ of $x_0$ such that
	\[
	\tau_a(U)\cap\tau_b(U)=\varnothing
	\qquad(a\ne b).
	\]
	Then
	\[
	p_N^{-1}(p_N(U))
	=
	\bigsqcup_{a\in G}\tau_a(U).
	\]
The quotient charts give $W'$ a smooth structure for which each restriction
	\[	p_N|_{\tau_a(U)}:\tau_a(U)\longrightarrow p_N(U)
	\]
	is a diffeomorphism. Hence $p_N$ is a smooth four-sheeted covering. This argument applies at every point of $N$, including the new core circles, so the covering is unbranched.	
	Moreover,
	\[
	p_N\circ\tau_a=p_N,
	\qquad
	p_N^{-1}(p_N(x_0))
	=
	\{\tau_a(x_0):a\in G\}.
	\]
	Thus the $\tau_a$ are deck transformations and act
	transitively on each fiber, proving regularity.
	If $d\in\operatorname{Deck}(p_N)$, there is a unique $a\in G$
	such that $d(x_0)=\tau_a(x_0)$.
	Since $N$ is connected, uniqueness of lifts gives $d=\tau_a$.
	Therefore,
	\[
	\operatorname{Deck}(p_N)
	=
	\{\tau_a:a\in G\}\cong G.
	\]
	
	Finally, each restriction
	\[
	r|_{\mathcal M_p}:\mathcal M_p\longrightarrow r(\mathcal M_0)
	\]
	is a diffeomorphism preserving the tube coordinates. Hence
	\[
	\left(
	M\setminus\bigcup_{p\in P}\operatorname{int}\mathcal M_p
	\right)/G
	\cong	(M/G)\setminus\operatorname{int}r(\mathcal M_0),
	\]
	and
	\[
	\left(
	\bigsqcup_{p\in P}\widetilde{\mathcal M}_p
	\right)/G
	\cong
	\widetilde{\mathcal M}_0.
	\]
	The compatibility of the gluing maps with the action gives
	\[
	W'
	\cong	\bigl((M/G)\setminus\operatorname{int}r(\mathcal M_0)\bigr)
	\cup_{\bar g}\widetilde{\mathcal M}_0,
	\]
	where
	\[
	\bar g
	=
	g_0\circ
	\bigl(r|_{\partial\mathcal M_0}\bigr)^{-1}
	:
	\partial r(\mathcal M_0)
	\longrightarrow
	\partial\widetilde{\mathcal M}_0.
	\]
	
	By Proposition~\ref{prop:quotient-original}, $r(\mathcal M_0)$ is a tubular neighborhood of the suspension
	orbit over $0\in\Sigma$.
	In the inherited boundary coordinates, $\bar g$ has the same coordinate expression as $g_0$.
	Thus Proposition~\ref{prop:top-slope} gives
	\[
	(\bar g^{-1})_*[m_{\mathrm{new}}]
	=
	[\ell_{\mathrm{old}}],
	\qquad
	(\bar g^{-1})_*[\ell_{\mathrm{new}}]
	=
	-[m_{\mathrm{old}}].
	\]
	Consequently, the quotient is obtained by a single filling
	along the same primitive degeneracy slope, with no meridional twist.
\end{proof}

\subsection{Identification of the quotient}

We identify $W'=N/G$ with the Hantzsche--Wendt manifold using two classical results on cyclic branched covers of the figure-eight knot.

\begin{proposition}\label{prop:top-figure-eight}
	\cite[Example~2.1, proof of Corollary~1.3, and Convention~2.4]{Zung} The complement of an open tubular neighborhood of the suspension orbit over $0\in\Sigma$ in
	$M_1(A;\Sigma)$ is homeomorphic to the figure-eight knot exterior.
	
	For every integer $n\geq1$, filling the complement of an open tubular neighborhood of the suspension
	orbit over $0$ in $M_1(A^n;\Sigma)$ along the degeneracy slope gives the $n$-fold cyclic cover of $S^3$ branched over the figure-eight knot.
\end{proposition}

\begin{proposition}
	\label{prop:top-HW-identification}
	\cite[p.~602, paragraph following Theorem~1.2]{Zimmermann}
	The three-fold cyclic cover of $S^3$ branched over the figure-eight knot is homeomorphic to the orientable flat
	Hantzsche--Wendt manifold $W$.
\end{proposition}

\begin{proposition}
	\label{prop:top-Wprime}
	The manifolds \(W'\) and \(W\) are diffeomorphic.
\end{proposition}

\begin{proof}
	By Propositions~\ref{prop:quotient-original} and~\ref{prop:top-roof}, $M/G$ is diffeomorphic to
	$M_1(B;\Sigma)$. Since $\rho=1$ near $0$, this diffeomorphism preserves the chosen tube coordinates near the suspension orbit over $0$.
	
	By Propositions~\ref{prop:top-regular}
	and~\ref{prop:top-slope}, $W'$ is obtained by filling the exterior of this orbit along the primitive
	degeneracy slope represented by $\ell_{\mathrm{old}}$, with the framing induced from $A$ under the
	three-sheeted suspension covering.
	
	Since $B=A^3$, Proposition~\ref{prop:top-figure-eight} with $n=3$ identifies $W'$ with the three-fold cyclic cover of $S^3$ branched over the figure-eight knot. Proposition~\ref{prop:top-HW-identification} then shows that $W'$ is homeomorphic to $W$.
	The three-dimensional Hauptvermutung
	\cite[Chapter~36, Theorem~2]{Moise}
	and the uniqueness of smoothings in dimension three
	\cite[Theorem~3.10.9]{Thurston}
	imply that $W'$ and $W$ are diffeomorphic.
\end{proof}

\subsection{The translation cover of the Hantzsche--Wendt manifold}

We identify the covering constructed in
Proposition~\ref{prop:top-regular} with the translation cover of the Hantzsche--Wendt manifold.

\begin{proposition}
	\label{prop:top-HW-group}
	The Hantzsche--Wendt manifold admits a flat realization
	\[
	W=\R^3/\Pi,
	\]
	where $\Pi$ is a Bieberbach group with presentation
	\[
	\Pi=
	\left\langle a,b,c\ \middle|\
	ab^2a^{-1}b^2=1,\;
	ba^2b^{-1}a^2=1,\;
	abc=1
	\right\rangle.
	\]
	Its translation subgroup $L$ is a full-rank lattice in $\R^3$ and satisfies
	\[
	L\triangleleft\Pi,
	\qquad
	\Pi/L\cong(\Z/2\Z)^2.
	\]
	Moreover,
	\[
	\Pi_{\mathrm{ab}}
	:=\Pi/[\Pi,\Pi]
	\cong(\Z/4\Z)^2.
	\]
\end{proposition}

\begin{proof}
	The flat realization, the presentation of $\Pi$,
	and the structure of its translation subgroup are classical; see
	\cite[Chapter~I, Example~2.4 and Exercise~2.2;
	Chapter~V, Section~6, Example~6.1]{Charlap}.
	
	It remains to compute the abelianization.
	Writing the group operation additively in $\Pi_{\mathrm{ab}}$, the defining relations give
	\[
	4[b]=0,\qquad
	4[a]=0,\qquad
	[a]+[b]+[c]=0.
	\]
	Eliminating $[c]$ yields
	\[
	\Pi_{\mathrm{ab}}
	=
	\langle [a],[b]\mid
	4[a]=0,\ 4[b]=0\rangle
	\cong(\Z/4\Z)^2.
	\]
\end{proof}

\begin{proposition}
	\label{prop:top-torus}
	Let $h:W'\to W$ be a diffeomorphism as in Proposition~\ref{prop:top-Wprime}.
	Then the covering
	\[
	h\circ p_N:N\longrightarrow W
	\]
	is smoothly equivalent to the natural projection
	\[
	p_L:\R^3/L\longrightarrow\R^3/\Pi=W.
	\]
	In particular, $N$ is diffeomorphic to $\mathbb T^3$.
\end{proposition}

\begin{proof}
	Set $p=h\circ p_N$.
	By Proposition~\ref{prop:top-regular}, $p$ is a
	connected smooth regular four-sheeted covering with deck group $(\Z/2\Z)^2$.
	Let $H\leq\Pi$ be the subgroup corresponding to $p$. By
	\cite[Proposition~1.39(a),(b)]{Hatcher},
	\[
	H\triangleleft\Pi,
	\qquad
	\Pi/H\cong(\Z/2\Z)^2.
	\]
	
	Since both $\Pi/H$ and $\Pi/L$ are abelian groups of exponent two, their quotient maps factor through
	\[
	q_2:\Pi\longrightarrow
	\Pi_{\mathrm{ab}}/2\Pi_{\mathrm{ab}}.
	\]
	By Proposition~\ref{prop:top-HW-group},
	\[
	\Pi_{\mathrm{ab}}/2\Pi_{\mathrm{ab}}
	\cong(\Z/2\Z)^2.
	\]
	The induced quotient maps are surjective homomorphisms between groups of order four, and hence are isomorphisms. Therefore,
	\[
	H=\ker q_2=L.
	\]
	
	By the classification of covering spaces \cite[Theorem~1.38]{Hatcher}, there is a covering equivalence
	\[
	F:N\longrightarrow\R^3/L
	\]
	satisfying
	\[
	p_L\circ F=p.
	\]
	Since both covering maps are smooth local diffeomorphisms, $F$ is a local diffeomorphism. Being a covering equivalence, $F$ is therefore
	a diffeomorphism.
	
	Finally, since $L$ is a full-rank latticein $\R^3$, there is a linear isomorphism
	$J:\R^3\to\R^3$ with $J(L)=\Z^3$.
	It induces a diffeomorphism
	\[
	\R^3/L\longrightarrow
	\R^3/\Z^3=\mathbb T^3.
	\]
	Thus $N$ is diffeomorphic to $\mathbb T^3$.
\end{proof}

\section{Quantitative transfer of mixing}\label{sec:transfer}

We now formulate the general transfer result used to pass exponential mixing through the singular conjugacy constructed above.

\begin{theorem}\label{thm:transfer}
Let $\varphi^t$ and $\Psi^t$ be complete smooth flows on $M$ and $N$, preserving probabilities $\mu$ and $\nu$. Let $S\subset M$ be a nonempty closed embedded smooth submanifold without boundary, and let $K\subset N$
be closed. Assume that:

\begin{enumerate}
\item
The sets $S$ and $K$ are invariant and satisfy $\mu(S)=\nu(K)=0.$

\item
There is a $C^1$ diffeomorphism
$\Phi:M\setminus S\longrightarrow N\setminus K
$ such that $\Phi\circ\varphi^t=\Psi^t\circ\Phi,$ and $\nu=\Phi_*\mu.$

\item There exist $\kappa>0$, $C_1>0$, and $r_*>0$ such that
\begin{equation}\label{eq:intro-small-mass}
\mu\{x:d_M(x,S)<r\}
\le C_1 r^\kappa,
\quad 0<r<r_*.
\end{equation}

\item There exist $p\ge1$ and $C_2>0$ such that
\begin{equation}\label{eq:intro-derivative}
\|D\Phi_x\|
\le
C_2d_M(x,S)^{-p},
\quad x\in M\setminus S.
\end{equation}
\end{enumerate}

Fix $0<\alpha<1$. Suppose that there exist
$C_\alpha,\gamma_\alpha>0$ such that
\begin{equation}\label{ineq:seed estimate}
|\Corr_{\mu,\varphi}(h,l;t)|
\le
C_\alpha e^{-\gamma_\alpha t}
\|h\|_{C^\alpha(M)}
\|l\|_{C^\alpha(M)}
\end{equation}
for all $h,l\in C^\alpha(M)$ and $t\ge0$.
Then there exists $C'_\alpha>0$ such that
\begin{equation}\label{eq:intro-transfer-estimate}
|\Corr_{\nu,\Psi}(f,g;t)|
\le
C'_\alpha
e^{-\kappa\gamma_\alpha t/(\kappa+2p\alpha)}
\|f\|_{C^\alpha(N)}
\|g\|_{C^\alpha(N)}
\end{equation}
for all $f,g\in C^\alpha(N)$ and $t\ge0$.
\end{theorem}

The rest of this section is devoted to the proof of Theorem~\ref{thm:transfer}. We write $d_S(x)=d_M(x,S)$ and fix $0<\alpha<1$. Constants may depend on the geometric data and the constants in the hypotheses, but not on the observables, time or cutoff scale.

\begin{lemma}\label{lem:transfer-cutoff}
There exists $\delta_*>0$ such that for every $0<\delta<\delta_*$,
there exists $\chi_\delta\in C^\infty(M;[0,1])$ satisfying
\begin{equation*}
 \chi_\delta=0\text{ on }\{d_S\le\delta\},\qquad
 \chi_\delta=1\text{ on }\{d_S\ge2\delta\},
\end{equation*}
and
\begin{equation*}
 \Lip(\chi_\delta)\le C\delta^{-1},\qquad
 [\chi_\delta]_{C^\alpha}\le C_\alpha\delta^{-\alpha}.
\end{equation*}
\end{lemma}

\begin{proof}
By the tubular neighborhood theorem, there exists \(r_0>0\) such that the normal exponential map
\[
\exp^\perp:
\{(y,v)\in NS:\ |v|<r_0\}
\longrightarrow
U_{r_0}(S):=\{x\in M:d_S(x)<r_0\}
\]
is a diffeomorphism. In these normal coordinates,
\[
d_S(\exp_y^\perp v)=|v|.
\]
Hence \(d_S\) is smooth on \(U_{r_0}(S)\setminus S\), while \(d_S^2\) is smooth on all of \(U_{r_0}(S)\).

Choose
\[
\zeta\in C^\infty([0,\infty);[0,1])
\]
such that $\zeta\equiv0 $ on $[0,1]$ and $
\zeta\equiv1$ on $[2,\infty)$. 
Fix $0<\delta_*<\frac{r_0}{2}.$
For \(0<\delta<\delta_*\), define
\[
\chi_\delta(x)=
\begin{cases}
\zeta\!\left(d_S(x)/\delta\right),
& x\in U_{r_0}(S),\\[1mm]
1,
& x\notin U_{r_0}(S).
\end{cases}
\]
Since \(\zeta\equiv0\) on \([0,1]\), one has $\chi_\delta\equiv0$ on $\{d_S\le\delta\}.$ Thus \(\chi_\delta\) is constant on a neighborhood of \(S\), so the possible nonsmoothness of \(d_S\) at \(S\) does not affect the smoothness of \(\chi_\delta\). Moreover, since \(2\delta<r_0\) and
\(\zeta\equiv1\) on \([2,\infty)\),
$\chi_\delta\equiv1$ on $\{2\delta\le d_S<r_0\}.$ Hence the two definitions above agree on a neighborhood of
\(\partial U_{r_0}(S)\), and therefore
$\chi_\delta\in C^\infty(M).$

Now we estimate the Lipschitz norm of $\chi_\delta$. On the transition region
\[
A_\delta
:=
\{x\in M:\delta<d_S(x)<2\delta\},
\]
the distance function \(d_S\) is smooth. Hence, for \(x\in A_\delta\), the chain rule gives
\begin{equation}\label{eq:chain rule for dchidelta}
D\chi_\delta(x)
=
\frac{1}{\delta}
\zeta'\!\left(\frac{d_S(x)}{\delta}\right)
Dd_S(x).
\end{equation}
We now use the fact that the distance to a closed set is \(1\)-Lipschitz:
\[
|d_S(x)-d_S(y)|
\le
d_M(x,y),
\quad x,y\in M.
\]
Since \(d_S\) is smooth at every point of \(A_\delta\), this implies
\[
\|Dd_S(x)\|\le 1,
\quad x\in A_\delta.
\]
Therefore, by \eqref{eq:chain rule for dchidelta},
\[
\|D\chi_\delta(x)\|
\le
\frac{\|\zeta'\|_{L^\infty}}{\delta}
\|Dd_S(x)\| 
\le \frac{\|\zeta'\|_{L^\infty}}{\delta},
\quad x\in A_\delta.
\]
Since \(\zeta'\) vanishes outside \((1,2)\), \(D\chi_\delta\) vanishes outside \(A_\delta\) as well. Therefore,
\[
\|D\chi_\delta\|_{L^\infty(M)}
\le
\|\zeta'\|_{L^\infty}\delta^{-1}.
\]
In particular,
\[
\operatorname{Lip}(\chi_\delta)
\le
\|D\chi_\delta\|_{L^\infty(M)}
\le
C\delta^{-1},
\]
where \(C=\|\zeta'\|_{L^\infty}\) is independent of \(\delta\).

Finally, since \(0\le\chi_\delta\le1\), we have
\[
|\chi_\delta(x)-\chi_\delta(y)|
\le
\min\left\{
1,\,
C\delta^{-1}d_M(x,y)
\right\}
\le C_\alpha\delta^{-\alpha}d_M(x,y)^\alpha,
\]
where the second inequality follows from the fact
\[
\min\{1,a\}\le a^\alpha,
\quad a\ge0, \quad 0<\alpha<1.
\]
Thus, we have 
\[
[\chi_\delta]_{C^\alpha}
\le
C_\alpha\delta^{-\alpha},
\]
which completes the proof.
\end{proof}

For $f\in C^\alpha(N)$, define $h=f\circ\Phi$ off $S$, and set $h=0$ on $S$. This gives a bounded measurable function on $M$. Let $h_\delta=\chi_\delta h$.

\begin{lemma}\label{lem:transfer-norms}
Let \(\kappa>0\) and \(p\ge1\) be the exponents in \eqref{eq:intro-small-mass} and \eqref{eq:intro-derivative}. There exist \(C>0\) and \(0<\delta_0<\min\{\delta_*,r_*/2,1\}\) such that for every \(0<\delta\le\delta_0\),
\begin{enumerate}
    \item [(1)] \( \|h_\delta\|_\infty\le\|f\|_\infty, \) \label{eq:transfer-sup}

    \item[(2)] \( \|h-h_\delta\|_{L^1(\mu)}\le C\delta^\kappa\|f\|_\infty,\) \label{eq:transfer-L1}

    \item[(3)] \(\|h_\delta\|_{C^\alpha(M)}\le
       C\delta^{-p\alpha}\|f\|_{C^\alpha(N)}. \) 
       \label{eq:transfer-holder}
\end{enumerate}
\end{lemma}
\begin{proof}
The first bound follows from $0\le\chi_\delta\le1$. Since \(2\delta\le2\delta_0<r_*\), assumption \eqref{eq:intro-small-mass} gives
\[
\begin{aligned}
\|h-h_\delta\|_{L^1(\mu)}
&=
\int_M (1-\chi_\delta(x))|h(x)|\,d\mu(x)\\
&\le
\|f\|_\infty\,\mu\{d_S<2\delta\} \le
C_1(2\delta)^\kappa\|f\|_\infty \le
C\delta^\kappa\|f\|_\infty.
\end{aligned}
\]
To prove the H\"older estimate, fix $x,y\in M$ and let $d=d_M(x,y)>0$. If $d\ge\delta/2$, 
\begin{equation}\label{ineq:h delta large}
 |h_\delta(x)-h_\delta(y)|
 \le2\|f\|_\infty
 \le2^{1+\alpha}\delta^{-\alpha}\|f\|_\infty d^\alpha
 \le2^{1+\alpha}\delta^{-p\alpha}\|f\|_\infty d^\alpha.
\end{equation}
The final inequality uses $p\ge1$ and $\delta\le1$.

Now suppose $d<\delta/2$. If both cutoff values vanish, the difference is zero. We may assume $\chi_\delta(x)>0$. Then $d_S(x)>\delta$. Let $\sigma$ be a minimizing geodesic from $x$ to $y$ satisfying $\sigma(0)=x,
\sigma(d)=y, |\dot\sigma(t)|=1$ and
$\operatorname{length}(\sigma)=d.$ It exists because $M$ is compact. Every point $w$ of this geodesic satisfies
\begin{equation}\label{ineq:w distance}
 d_S(w)\ge d_S(x)-d_M(x,w)>\delta-d>\delta/2.
\end{equation}
The geodesic therefore lies in the domain of $\Phi$. Thus, we have
\begin{align*}
 d_N(\Phi x,\Phi y)
 & \le\operatorname{length}(\Phi\circ\sigma) =\int_0^d
\left|
\frac{d}{dt}\Phi(\sigma(t))
\right|_N\,dt \\
&\leq \int_0^d \|D\Phi_{\sigma(t)}\|
\,|\dot\sigma(t)|_M\,dt
 \le C\delta^{-p}d,
\end{align*}
where the last inequality holds due to the assumption \eqref{eq:intro-derivative} and \eqref{ineq:w distance}. We then deduce that
\begin{align*}
 |h_\delta(x)-h_\delta(y)|
 &\le\chi_\delta(x)|h(x)-h(y)|
      +|\chi_\delta(x)-\chi_\delta(y)|\,|h(y)|\\
 &\le [f]_{C^\alpha} d_N^\alpha(\Phi x,\Phi y) + [\chi_\delta]_{C^\alpha} d_M^\alpha(x,y) \|h\|_{\infty} \\
 &\le C[f]_{C^\alpha}\delta^{-p\alpha}d^\alpha
      +C\|f\|_\infty\delta^{-\alpha}d^\alpha
\le C\delta^{-p\alpha}\|f\|_{C^\alpha(N)}d^\alpha.
\end{align*}
Together with \eqref{ineq:h delta large}, the third bound holds.
\end{proof}

We now transfer the correlation function for the flow \(\Psi^t\) on \(N\) to the original flow \(\varphi^t\) on \(M\).

\begin{lemma}\label{lem:transfer-error}
Let $f,g\in C^\alpha(N)$ satisfy $\nu(f)=\nu(g)=0$. Define $h=f\circ\Phi$, $l=g\circ\Phi$ off $S$, and extend both by zero on $S$. Let
$h_\delta=\chi_\delta h$ and $l_\delta=\chi_\delta l$. Then, for every \(0<\delta\le\delta_0\) and every $t\ge0$,
\begin{equation}\label{eq:transfer-correlation-identity}
 \Corr_{\nu,\Psi}(f,g;t)=\Corr_{\mu,\varphi}(h,l;t),
\end{equation}
and there exists $C>0$, independent of $t$, such that 
\begin{equation}\label{ineq:correlation gap}
 |\Corr_{\mu,\varphi}(h,l;t)
   -\Corr_{\mu,\varphi}(h_\delta,l_\delta;t)|
 \le C\delta^\kappa\|f\|_\infty\|g\|_\infty,
\end{equation}
where $\kappa>0$ is the exponent in \eqref{eq:intro-small-mass}.
\end{lemma}

\begin{proof}
Since \(\nu=\Phi_*\mu\) and \(\mu(S)=\nu(K)=0\), it follows that $\mu(h)=\mu(l)=0$. Hence,
\[
\begin{aligned}
\operatorname{Corr}_{\mu,\varphi}(h,l;t)
&=
\int_M h(x)\,l(\varphi^t x)\,d\mu(x)=
\int_{M\setminus S}
f(\Phi x)\,
g\!\left(\Phi(\varphi^t x)\right)\,d\mu(x)\\
&=
\int_{M\setminus S}
f(\Phi x)\,
g\!\left(\Psi^t(\Phi x)\right)\,d\mu(x)=
\int_{N\setminus K}
f(y)\,g(\Psi^t y)\,d\nu(y)\\
&=
\int_N
f(y)\,g(\Psi^t y)\,d\nu(y)=
\operatorname{Corr}_{\nu,\Psi}(f,g;t).
\end{aligned}
\]
For \eqref{ineq:correlation gap}, we have
\begin{align}\label{eq:correlation split}
&\operatorname{Corr}_{\mu,\varphi}(h,l;t)
-
\operatorname{Corr}_{\mu,\varphi}(h_\delta,l_\delta;t) \notag
\\
&\quad=
\int_M
h(x)l(\varphi^t x)\,d\mu(x)
-
\int_M
h_\delta(x)l_\delta(\varphi^t x)\,d\mu(x)
+
\left(\int_M h_\delta\,d\mu\right)
\left(\int_M l_\delta\,d\mu\right) \notag
\\
&\quad=
\int_M
h(x)l(\varphi^t x)
\Bigl[
1-\chi_\delta(x)\chi_\delta(\varphi^t x)
\Bigr]
\,d\mu(x)
+
\left(\int_M h_\delta\,d\mu\right)
\left(\int_M l_\delta\,d\mu\right). 
\end{align}
Since \(0\le \chi_\delta\le1\),
\[
1-\chi_\delta(x)\chi_\delta(\varphi^t x)
\le
\bigl(1-\chi_\delta(x)\bigr)
+
\bigl(1-\chi_\delta(\varphi^t x)\bigr).
\]
Hence
\[
\begin{aligned}
&
\left|
\int_M
h(x)l(\varphi^t x)
\Bigl[
1-\chi_\delta(x)\chi_\delta(\varphi^t x)
\Bigr]
\,d\mu(x)
\right|
\\
&\qquad\le
\|h\|_\infty\|l\|_\infty
\left[
\int_M (1-\chi_\delta(x))\,d\mu(x)
+
\int_M (1-\chi_\delta(\varphi^t x))\,d\mu(x)
\right].
\end{aligned}
\]
By the \(\varphi^t\)-invariance of \(\mu\),
\[
\int_M (1-\chi_\delta(\varphi^t x))\,d\mu(x)
=
\int_M (1-\chi_\delta(x))\,d\mu(x)
\leq \mu(U_{2\delta}(S))
\le
C\delta^\kappa.
\]
Here \(1-\chi_\delta\) vanishes outside \(U_{2\delta}(S)\), and the last inequality follows from \eqref{eq:intro-small-mass}. Consequently,
\begin{equation}\label{ineq:correlation sum 1}
\left|
\int_M
h(x)l(\varphi^t x)
\Bigl[
1-\chi_\delta(x)\chi_\delta(\varphi^t x)
\Bigr]
\,d\mu(x)
\right|
\le
C\delta^\kappa
\|f\|_\infty\|g\|_\infty.
\end{equation}
For the second term in \eqref{eq:correlation split}, 
\[
\left|\int_M h_\delta\,d\mu\right|
=
\left|
\int_M(\chi_\delta-1)h\,d\mu
\right|\le
\|h\|_\infty
\mu(U_{2\delta}(S))
\le
C\delta^\kappa\|h\|_\infty.
\]
The same estimate holds for $\mu(l_\delta)$. Thus,
\begin{equation}\label{ineq:h delta mean}
\left|
\left(\int_M h_\delta\,d\mu\right)
\left(\int_M l_\delta\,d\mu\right)
\right|
\le
C\delta^{2\kappa}
\|h\|_\infty\|l\|_\infty
\le
C\delta^\kappa
\|f\|_\infty\|g\|_\infty.
\end{equation}
Combining \eqref{eq:correlation split}, \eqref{ineq:correlation sum 1} and \eqref{ineq:h delta mean}, \eqref{ineq:correlation gap} holds. The constant is independent of $t$.
\end{proof}

Now we can prove the general transfer theorem as stated in Theorem~\ref{thm:transfer}.
\begin{proof}[Proof of Theorem~\ref{thm:transfer}]
We may assume $\nu(f)=\nu(g)=0$. Let $h_\delta,l_\delta$ be defined as in Lemma~\ref{lem:transfer-error}. 
By \eqref{eq:transfer-correlation-identity} and \eqref{ineq:correlation gap} we have
\begin{align}\label{ineq:correlation final estimate}
|\operatorname{Corr}_{\nu,\Psi}(f,g;t)|
&\le
\left|
\operatorname{Corr}_{\mu,\varphi}(h,l;t)
-
\operatorname{Corr}_{\mu,\varphi}(h_\delta,l_\delta;t)
\right|
+
|\operatorname{Corr}_{\mu,\varphi}(h_\delta,l_\delta;t)| \notag \\
& \leq C\delta^\kappa
\|f\|_\infty\|g\|_\infty + Ce^{-\gamma_\alpha t}
\|h_\delta\|_{C^\alpha(M)}
\|l_\delta\|_{C^\alpha(M)} \notag \\
& \leq C\left(
\delta^\kappa
+
e^{-\gamma_\alpha t}\delta^{-2p\alpha}
\right)
\|f\|_{C^\alpha(N)}
\|g\|_{C^\alpha(N)}.
\end{align}
The second inequality holds due to \eqref{ineq:seed estimate} and the third follows from Lemma~\ref{lem:transfer-norms}. Choose a fixed admissible $\delta_0\le1$ and put
\begin{equation*}
 \delta=\delta_0
   \exp\left(-\frac{\gamma_\alpha}{\kappa+2p\alpha}t\right).
\end{equation*}
The two terms in
\eqref{ineq:correlation final estimate} become:
\begin{align*}
 \delta^\kappa
 &=\delta_0^\kappa
   e^{-\kappa\gamma_\alpha t/(\kappa+2p\alpha)},\\
 e^{-\gamma_\alpha t}\delta^{-2p\alpha}
 &=\delta_0^{-2p\alpha}
   e^{-\kappa\gamma_\alpha t/(\kappa+2p\alpha)}.
\end{align*}
Absorb their fixed prefactors into $C'_\alpha$. This is exactly
\eqref{eq:intro-transfer-estimate}.
\end{proof}

\section{Completion on the standard three-torus}\label{sec:completion}

Section~\ref{sec:transfer} provides an abstract criterion for transferring exponential mixing through the singular conjugacy. We now verify that the present surgery satisfies the two quantitative
hypotheses of that criterion. This yields
exponential mixing for the smooth flow on \(N\).

It then remains to pass from the abstract manifold \(N\) to $\mathbb{T}^3$. Using the diffeomorphism \(N\cong\mathbb T^3\) obtained in Section~\ref{sec:topology} and a volume-normalizing smooth change of coordinates, we transport the flow to \(\mathbb T^3\) with normalized Lebesgue volume.
This completes the construction of the vector field in Theorem~\ref{thm:main}.

\subsection{Verification of the transfer assumptions}
\begin{proposition}
\label{prop:completion-bounds}
The surgery map satisfies the hypotheses
\eqref{eq:intro-small-mass}, \eqref{eq:intro-derivative} with
$S=\Gamma,\kappa=2$ and $p=1.$
\end{proposition}

\begin{proof}
For \(p\in P\) and \(0<R'<R\), write
\[
\mathcal M_p(R')
=
\bigl\{(x,y,s)\text{ in the tube at }p:
x^2+y^2<R'^2\bigr\}.
\]
Choose $R_0,R_1$ such that
\[
0<R_1<\frac{R_0}{2}<R_0<R.
\]
Hence the closures of the four tubes
\(\mathcal M_p(R_0)\) are pairwise disjoint. For \(0<\delta<R_1\), let $U_\delta=\bigcup_{p\in P} \mathcal M_p(\delta)$. By \eqref{eq:solid-torus-normal-form} and the normalization in \eqref{eq:seed-volume-normalization}, we have $\mu(U_\delta)=\frac{4\pi}{Z}\delta^2$.

To obtain \eqref{eq:intro-small-mass}, we need to establish the bound between $r(q)=(x^2+y^2)^{1/2}$ and $d_M(q,\Gamma)$ for $q$ in these tubes. We claim that there exist $c_0>0,C_0>0$ such that 
\[
c_0 r(q) \leq d_M(q,\Gamma) \leq C_0 r(q).
\]
Let \(g_M\) and \(g_N\) denote the fixed smooth Riemannian metrics on \(M\) and \(N\), respectively. Let $g_{0}=dx^2+dy^2+ds^2$ denote the product metric in the tube coordinates on $M$. Since $\bigcup_{p\in P}\overline{\mathcal M_p(R_0)}$ is compact and the metric \(g_M\) is smooth and positive definite, there exist constants $0<c_0\le C_0<\infty$ such that, for every tangent vector \(v\) based in these tubes,
\[
c_0 |v|_{g_0}
\le
|v|_{g_M}
\le
C_0 |v|_{g_0}.
\]
Consequently, for every piecewise \(C^1\) curve \(\sigma\) contained in one of these tubes,
\[
c_0 L_{g_0}(\sigma)
\le
L_{g_M}(\sigma)
\le
C_0 L_{g_0}(\sigma).
\]
Here \(L_g(\sigma)\) denotes the length of the curve \(\sigma\) with respect to the Riemannian metric \(g\).
Fix $q=(x,y,s)\in \mathcal M_p(R_1),$
consider the radial curve
\[
\sigma_{\mathrm{rad}}(u)
:=
((1-u)x,(1-u)y,s),
\qquad
0\le u\le1.
\]
Then $\sigma_{\mathrm{rad}}(0)=q,
\sigma_{\mathrm{rad}}(1)\in\Gamma_p$ and $
L_{g_0}(\sigma_{\mathrm{rad}})=r.$
Therefore
\[
d_M(q,\Gamma)
\le
L_{g_M}(\sigma_{\mathrm{rad}})
\le
C_0 r.
\]
For the reverse inequality, let $\sigma$ be any piecewise \(C^1\) curve satisfying $\sigma(0)=q, \sigma(1)\in\Gamma$.
If $\sigma([0,1])\subset \mathcal M_p(R_0)$, writing
\[
\sigma(u)
=
(x(u),y(u),s(u)),
\qquad
w(u)=(x(u),y(u)),
\]
we have $|w(0)|=r,w(1)=0.$ Thus, we have
\[
L_{g_M}(\sigma)
\ge
c_0L_{g_0}(\sigma)
\ge
c_0\int_0^1 |\dot w(u)|\,du
\ge
c_0|w(1)-w(0)|
=c_0r.
\]
If \(\sigma\) leaves \(\mathcal M_p(R_0)\), define
\[
u_*=
\inf\{u\in[0,1]:
\sigma(u)\notin \mathcal M_p(R_0)\}.
\]
Then $|w(u_*)|=R_0$. Therefore,
\[
\begin{aligned}
L_{g_M}(\sigma)
&\ge
L_{g_M}\bigl(\sigma|_{[0,u_*]}\bigr)
\ge
c_0
\int_0^{u_*}|\dot w(u)|\,du \\
&\ge
c_0|w(u_*)-w(0)|
\geq
c_0(R_0-r)
\ge
c_0r,
\end{aligned}
\]
where the last inequality follows from
$r<R_1<\frac{R_0}{2}.$ Therefore,
\[
c_0r
\le
d_M(q,\Gamma)
\le
C_0r
\quad
\text{for }
q\in
\bigcup_{p\in P}\mathcal M_p(R_1).
\]

Set
\[
E
=
M\setminus
\bigcup_{p\in P}\mathcal M_p(R_1).
\]
Since \(E\) and \(\Gamma\) are disjoint compact sets,
\[
d_*
:=
d_M(E,\Gamma)
>0.
\]
Thus, for \(0<\varepsilon< \min \{d_*,c_0 R_1\}\), if $d_M(q,\Gamma)<\varepsilon$, then $q\in
\bigcup_{p\in P}\mathcal M_p(R_1)$. Using the lower bound above, we have $c_0r(q)<\varepsilon$. Therefore, 
\[
\mu\{q\in M:
d_M(q,\Gamma)<\varepsilon\}
\le
\mu(U_{\varepsilon/c_0})
=
\frac{4\pi}{Zc_0^2}\varepsilon^2.
\]
Thus \eqref{eq:intro-small-mass} holds with $\kappa=2.$

We next estimate \(D\Phi_q\) for $q \in M\setminus\Gamma$. For $q \in \mathcal M_p(R_1)\setminus\Gamma_p$, let $g_{0,N}=d\xi^2+d\eta^2+dz^2$ be the product metric in the new tube coordinates. For any $v\in T_qM,$ we have
\[
D\Phi_q(v)
=
d\xi(v)\,\partial_\xi
+
d\eta(v)\,\partial_\eta
+
dz(v)\,\partial_z,
\]
and hence
\begin{equation}\label{eq:dphiq}
|D\Phi_q(v)|_{g_{0,N}}^2
=
|d\xi(v)|^2
+
|d\eta(v)|^2
+
|dz(v)|^2.
\end{equation}
On a punctured old tube,
\[
r=(x^2+y^2)^{1/2},
\qquad
z=-\arg(x+iy),
\]
and
\[
\xi=r\cos(2\pi s),
\qquad
\eta=r\sin(2\pi s).
\]
Differentiation gives
\[
dr
=
\frac{x\,dx+y\,dy}{r},
\qquad
dz
=
\frac{y\,dx-x\,dy}{r^2},
\]
and
\[
d\xi
=
\cos(2\pi s)\,dr
-
2\pi r\sin(2\pi s)\,ds,
\]
\[
d\eta
=
\sin(2\pi s)\,dr
+
2\pi r\cos(2\pi s)\,ds.
\]
For $v=v_x\partial_x+v_y\partial_y+v_s\partial_s$, by Cauchy--Schwarz inequality, we have
\[
|dr(v)|
=
\left|
\frac{xv_x+yv_y}{r}
\right|
\le
(v_x^2+v_y^2)^{1/2}
\le
|v|_{g_0},
\]
and
\[
|dz(v)|
=
\left|
\frac{yv_x-xv_y}{r^2}
\right|
\le
\frac{(v_x^2+v_y^2)^{1/2}}{r}
\le
\frac{|v|_{g_0}}{r}.
\]
Moreover,
\[
|d\xi(v)|
\le
|dr(v)|+2\pi r|v_s|,
\quad
|d\eta(v)|
\le
|dr(v)|+2\pi r|v_s|.
\]
Substituting these estimates into \eqref{eq:dphiq} and shrinking \(R_1\) if necessary so that \(R_1\le1\), we have
\begin{equation}\label{ineq:dphiq}
|D\Phi_q(v)|_{g_{0,N}}
\le
\frac{C_1}{r(q)}
|v|_{g_0},
\quad
q\in
\mathcal M_p(R_1)\setminus\Gamma_p.
\end{equation}

As in the distance comparison above, smooth Riemannian metrics are uniformly equivalent to the corresponding product metrics on the closures of these fixed coordinate tubes. Thus there exist constants \(c_M,C_N>0\) such that
\[
|v|_{g_{0}}
\le
c_M^{-1}|v|_{g_M},
\qquad
|w|_{g_N}
\le
C_N|w|_{g_{0,N}}.
\]
Combining these inequalities with \eqref{ineq:dphiq}, we obtain
\[
|D\Phi_q(v)|_{g_N}
\le
\frac{C_NC_1}{c_M\,r(q)}|v|_{g_M}.
\]
Taking the supremum over \(v\neq0\) therefore gives
\begin{equation}\label{ineq:dphiq near gamma}
\|D\Phi_q\|
\le
\frac{C_2}{r(q)} \leq C_3\,d_M(q,\Gamma)^{-1}, 
\quad
q\in \mathcal M_p(R_1)\setminus\Gamma_p,
\end{equation}
where the second inequality holds by $d_M(q,\Gamma)\le C_0r(q)$.

Finally, since \(\Phi\) is smooth on the compact set $M':=M\setminus
\bigcup_{p\in P}\mathcal M_p(R_1/2),$
\[
B:=
\sup
\left\{
\|D\Phi_q\|:
q\in M'
\right\}
<\infty.
\]
Hence, for every $q \in M'$,
\[
\|D\Phi_q\|
\le B 
\le
B\,\operatorname{diam}(M)\,
d_M(q,\Gamma)^{-1}.
\]
Combining with \eqref{ineq:dphiq near gamma}, \eqref{eq:intro-derivative} holds with $p=1.$
\end{proof}

Proposition~\ref{prop:seed-volume} gives the invariant probability measure $\mu$ for the seed flow.
Propositions~\ref{prop:complement-map},
\ref{prop:new-vector-field}, \ref{prop:new-volume}
and \ref{prop:same-time-conjugacy}, together with the invariance of $\Gamma$ and $K$ established in Section~\ref{sec:surgery}, verify the dynamical assumptions of Theorem~\ref{thm:transfer}.
In particular, $\mu(\Gamma)=\nu(K)=0$.
Proposition~\ref{prop:completion-bounds} verifies the two quantitative assumptions with $(\kappa,p)=(2,1)$.
Applying Theorem~\ref{thm:transfer} together with the seed estimate in Proposition~\ref{prop:seed-mixing}
yields the following corollary.
\begin{corollary}
\label{cor:completion-mixing}
The flow of $Y$ on $(N,\nu)$ is exponentially mixing. If
$\gamma_\alpha$ is the seed exponent in
Proposition~\ref{prop:seed-mixing}, then
\begin{equation*}
 |\Corr_{\nu,\Psi}(f,g;t)|
 \le C'_\alpha e^{-\gamma_\alpha t/(1+\alpha)}
\|f\|_{C^\alpha(N)}\|g\|_{C^\alpha(N)}.
\end{equation*}
\end{corollary}

\subsection{The final vector field}

The invariant probability $\nu$ on $N$ is smooth and positive, but an arbitrary diffeomorphism $N\to\mathbb T^3$ need not send it to normalized Lebesgue volume. We use Moser's volume-form theorem to correct this.

\begin{proposition} \cite[Theorem~2]{Moser}
\label{prop:volume-normalization}
Let $\omega_0$ and $\omega_1$ be smooth positive volume forms on
$\mathbb T^3$ satisfying
\[
\int_{\mathbb T^3}\omega_0
=
\int_{\mathbb T^3}\omega_1.
\]
Then there exists an orientation-preserving smooth diffeomorphism $F:\mathbb T^3\longrightarrow\mathbb T^3$
such that $F^*\omega_1=\omega_0.$ 
\end{proposition}

We now have all the ingredients needed for the final construction.
\begin{proposition}
\label{prop:completion-conjugacy}
There is an orientation-preserving diffeomorphism
$J:N\to\T^3$ such that $J_*\nu=m$. For any such $J$, define
\begin{equation}\label{eq:completion-final-field}
 X=J_*Y,\qquad \psi^t=J\circ\Psi^t\circ J^{-1}.
\end{equation}
Then $X$ is smooth and autonomous, and $\operatorname{div}_mX=0$.
It has exactly sixteen hyperbolic equilibria. Its flow satisfies \eqref{eq:intro-exponential} with
$c_\alpha=\gamma_\alpha/(1+\alpha)$.
\end{proposition}
\begin{proof}
By Proposition~\ref{prop:top-torus}, we can choose a smooth and orientation-preserving diffeomorphism $D:N\to\T^3$. Then $D_*\nu=\rho_1\,m$ with a smooth positive density $\rho_1$. Apply Proposition~\ref{prop:volume-normalization} with $\omega_0=\rho_1 m, \omega_1=m$ and set $J=F\circ D$. This gives $J_*\nu=m$.

Define the vector field on \(\mathbb T^3\) by $X:=J_*Y$. That is,
\begin{equation}\label{eq:vector field x}
X(Jy):=D J_y\,Y(y),
\quad y\in N.
\end{equation}
Since \(J\) and \(Y\) are smooth and \(J\) is a diffeomorphism,
\(X\) is a smooth autonomous vector field on \(\mathbb T^3\).

We first identify its flow. Define
$\psi^t
=
J\circ \Psi^t\circ J^{-1}.$ For \(x=Jy\), we have
\[
\frac{d}{dt}\psi^t(x)
=
\frac{d}{dt}J(\Psi^t y)
=
D J_{\Psi^t y}\,Y(\Psi^t y)
=
X(J(\Psi^t y))
=
X(\psi^t(x)).
\]
Since \(\psi^0(x)=x\), uniqueness of solutions shows that
\(\psi^t\) is precisely the flow generated by \(X\).

Next we show that \(X\) is divergence-free with respect to normalized Lebesgue volume \(m\). Since \(J_*\nu=m\) and \(\Psi^t\) preserves \(\nu\), we have
\[
\begin{aligned}
(\psi^t)_*m
&=
(J\circ\Psi^t\circ J^{-1})_*m
=
J_*(\Psi^t)_*(J^{-1})_*m
\\
&=
J_*(\Psi^t)_*\nu
=
J_*\nu
=m.
\end{aligned}
\]
Thus the flow of \(X\) preserves \(m\), which is equivalent to $\operatorname{div}_m X=0$.

It remains to compare the equilibria of $Y$ and $X$. Since $X(J(y))=DJ(y)Y(y)$
and $DJ(y)$ is invertible,
\[
X(J(y))=0
\quad\Longleftrightarrow\quad
Y(y)=0.
\]
Thus $J$ gives a bijection between the equilibria of $Y$ and those of $X$.
By Proposition~\ref{prop:new-equilibria}, $X$ therefore has exactly sixteen equilibria. Moreover, let $y$ be an equilibrium of $Y$. Working in local coordinates around \(y\) and \(J(y)\), differentiating \eqref{eq:vector field x} at \(y\) gives 
\[
DX(J(y))
=
DJ(y)\,DY(y)\,DJ(y)^{-1}.
\]
Hence \(DX(J(y))\) and \(DY(y)\) are similar and have the same spectrum. Thus hyperbolicity is preserved under the change of coordinates.

Finally, we transfer the exponential mixing estimate. Let
\(f,g\in C^\alpha(\mathbb T^3)\). Since \(J_*\nu=m\) and $\psi^t\circ J=J\circ\Psi^t,$ a change of variables gives
\[
\operatorname{Corr}_{m,\psi}(f,g;t)
=
\operatorname{Corr}_{\nu,\Psi}
(f\circ J,g\circ J;t).
\]
Since \(J\) is smooth on the compact manifold \(N\), it is Lipschitz.
Hence, there exists a constant \(C_J\) depending only on \(J\) such that
\[
\|f\circ J\|_{C^\alpha(N)}
\le
C_J\|f\|_{C^\alpha(\mathbb T^3)}
\]
The same estimate holds for \(g\). Applying Corollary~\ref{cor:completion-mixing} yields
\[
\begin{aligned}
\bigl|
\operatorname{Corr}_{m,\psi}(f,g;t)
\bigr|
&\le
C'_\alpha
e^{-\gamma_\alpha t/(1+\alpha)}
\|f\circ J\|_{C^\alpha(N)}
\|g\circ J\|_{C^\alpha(N)}
\\
&\le
C''_\alpha
e^{-\gamma_\alpha t/(1+\alpha)}
\|f\|_{C^\alpha(\mathbb T^3)}
\|g\|_{C^\alpha(\mathbb T^3)}.
\end{aligned}
\]
Thus the flow of \(X\) is exponentially mixing on \((\mathbb T^3,m)\), with decay exponent
$c_\alpha=\frac{\gamma_\alpha}{1+\alpha}.$
This proves the required estimate.
\end{proof}

\begin{proof}[Proof of Theorem~\ref{thm:main}]
Propositions~\ref{prop:seed-volume} and~\ref{prop:seed-mixing}
give the seed flow and its exponential correlation bound.
Propositions~\ref{prop:new-vector-field}--\ref{prop:new-equilibria}
construct the new smooth volume-preserving flow and its sixteen
hyperbolic equilibria. Proposition~\ref{prop:top-torus} identifies its manifold with $\T^3$. Corollary~\ref{cor:completion-mixing}
and Proposition~\ref{prop:completion-conjugacy} then give the claimed field on the standard torus with normalized Lebesgue volume.
\end{proof}

\appendix
\section{H\"older--Sobolev transfer for exponential mixing}
\label{app:regularity-transfer}

\begin{lemma}
\label{lem:holder-to-sobolev}
Let \(\mathbb T^d\) be equipped with its standard flat metric and normalized Lebesgue measure \(m\). Let
$\varphi^t:\mathbb T^d\to\mathbb T^d$ be a measurable measure-preserving flow. Fix $0<\alpha<1$ and assume that there exist
constants $C_\alpha>0,
\gamma_\alpha>0$
such that
\begin{equation}
\label{eq:holder-mixing-assumption}
\left|
\operatorname{Corr}_{m,\varphi}(F,G;t)
\right|
\le
C_\alpha e^{-\gamma_\alpha t}
\|F\|_{C^\alpha(\mathbb T^d)}
\|G\|_{C^\alpha(\mathbb T^d)}
\end{equation}
for every $F,G\in C^\alpha(\mathbb T^d)$ and every $t\ge0$.

Then, for every $s>0$, there exist constants $C_{s,\alpha,d}>0,
\gamma_{s,\alpha,d}>0$ such that
\begin{equation}
\label{eq:sobolev-mixing-conclusion}
\left|
\operatorname{Corr}_{m,\varphi}(f,g;t)
\right|
\le
C_{s,\alpha,d}
e^{-\gamma_{s,\alpha,d}t}
\|f\|_{H^s(\mathbb T^d)}
\|g\|_{H^s(\mathbb T^d)}
\end{equation}
for every $f,g\in H^s(\mathbb T^d)$ and every $t\ge0$.
\end{lemma}

\begin{proof}
Here, we adopt the following notation:
\[
\widehat f(k)
:=
\int_{\mathbb T^d}f(x)e^{-2\pi i k\cdot x}\,dm(x),
\quad
\|f\|_{H^s}^2
:=
\sum_{k\in\mathbb Z^d}
(1+4\pi^2|k|^2)^s|\widehat f(k)|^2.
\]
We may assume $
\int_{\mathbb T^d}f\,dm
=
\int_{\mathbb T^d}g\,dm
=0.$ For every $h\in L^2(m)$, define $U_t h:=h\circ\varphi^t.$ Then, for every $t \geq 0$, $\|U_t h\|_{L^2(m)}=\|h\|_{L^2(m)}$.
In particular,
\[
\operatorname{Corr}_{m,\varphi}(f,g;t)
=
\langle f,U_tg\rangle_{L^2(m)}.
\]
For $N\ge1$, define the Fourier truncation
\[
P_N f(x)
:=
\sum_{k\in\mathbb Z^d,|k|\le N}
\widehat f(k)e^{2\pi i k\cdot x}.
\]
Since $P_N$ leaves the zero mode unchanged, it preserves the mean. Let $f_N:=P_Nf, g_N:=P_Ng.$ Parseval's identity gives
\begin{equation}\label{eq:fourier-approximation}
\|f-f_N\|_{L^2}^2
=
\sum_{|k|>N}|\widehat f(k)|^2
\le
(2\pi N)^{-2s}
\sum_{|k|>N}
(2\pi|k|)^{2s}
|\widehat f(k)|^2
\le
(2\pi N)^{-2s}
\|f\|_{H^s}^2.
\end{equation}
The same estimate holds for $g$. Since
\[
\#\{k\in\mathbb Z^d:|k|\le N\}
\le
(2N+1)^d
\le
(3N)^d,
\qquad N\ge1,
\]
the Cauchy--Schwarz inequality and Parseval's identity imply
\begin{equation}\label{eq:fourier-Linfty}
\|f_N\|_{L^\infty}
\le
\sum_{|k|\le N}|\widehat f(k)|
\le
(3N)^{\frac{d}{2}}
(\sum_{|k|\le N}|\widehat f(k)|^2)^{1/2} \le
3^{\frac{d}{2}}N^{\frac{d}{2}}\|f\|_{L^2}.
\end{equation}
Similarly,
\begin{equation}\label{eq:fourier-gradient}
\|\nabla f_N\|_{L^\infty}
\le 2\pi
\sum_{|k|\le N}|k|\,|\widehat f(k)|
\le
2\pi\,3^{\frac{d}{2}}
N^{\frac{d}{2}+1}
\|f\|_{L^2}.
\end{equation}

For every $h \in C^1(\mathbb T^d)$, we have
\[
|h(x)-h(y)|
\le
\min
\left\{
2\|h\|_{L^\infty},
\|\nabla h\|_{L^\infty}d_{\mathbb{T}^d}(x,y)
\right\}.
\]
Since $
\min\{a,b\}
\le
a^{1-\alpha}b^\alpha$ holds for $a,b>0$ and $0<\alpha<1$, we deduce that
\begin{equation}
\label{eq:holder-interpolation}
[h]_{C^\alpha}
\le
2^{1-\alpha}
\|h\|_{L^\infty}^{1-\alpha}
\|\nabla h\|_{L^\infty}^{\alpha}.
\end{equation}
Applying \eqref{eq:fourier-Linfty} and
\eqref{eq:fourier-gradient} to
\eqref{eq:holder-interpolation} yields
\begin{equation*}
[f_N]_{C^\alpha}
\leq 
2^{1-\alpha}(2\pi)^\alpha
3^{\frac{d}{2}}
N^{\frac{d}{2}+\alpha}
\|f\|_{L^2}.
\end{equation*}
Combining this with \eqref{eq:fourier-Linfty}, we obtain
\begin{equation}
\label{eq:fourier-holder}
\|f_N\|_{C^\alpha}
\le
A_{d,\alpha}
N^{\frac{d}{2}+\alpha}
\|f\|_{L^2}
\leq A_{d,\alpha}
N^{\frac{d}{2}+\alpha}
\|f\|_{H^s},
\end{equation}
where $A_{d,\alpha}=3^{\frac{d}{2}}(
1+2^{1-\alpha}(2\pi)^\alpha
).$ The same estimates hold for $g_N$.

We now decompose the correlation. Using \eqref{eq:fourier-approximation}, we obtain
\begin{align*}
\left|
\langle f,U_tg\rangle_{L^2}
-
\langle f_N,U_tg_N\rangle_{L^2}
\right|
& = \left| \langle f-f_N,U_tg\rangle_{L^2}
+
\langle f_N,U_t(g-g_N)\rangle_{L^2}  \right| \nonumber\\
&\le
\|f-f_N\|_{L^2}\|g\|_{L^2}
+
\|f_N\|_{L^2}\|g-g_N\|_{L^2}
\nonumber\\
&\le
2(2\pi)^{-s}
N^{-s}
\|f\|_{H^s}
\|g\|_{H^s}.
\end{align*}
Hence
\begin{equation}
\label{eq:correlation-approx-error}
\left|
\operatorname{Corr}_{m,\varphi}(f,g;t)
-
\operatorname{Corr}_{m,\varphi}(f_N,g_N;t)
\right|
\le
2(2\pi)^{-s}
N^{-s}
\|f\|_{H^s}\|g\|_{H^s}.
\end{equation}
For the truncated functions, the assumed H\"older mixing estimate
\eqref{eq:holder-mixing-assumption} and
\eqref{eq:fourier-holder} give
\begin{align*}
\left|
\operatorname{Corr}_{m,\varphi}(f_N,g_N;t)
\right|
&\le
C_\alpha e^{-\gamma_\alpha t}
\|f_N\|_{C^\alpha}
\|g_N\|_{C^\alpha}
\nonumber\\
&\le
C_\alpha A_{d,\alpha}^2
e^{-\gamma_\alpha t}
N^{d+2\alpha}
\|f\|_{H^s}
\|g\|_{H^s}.
\end{align*}
Combining this with
\eqref{eq:correlation-approx-error}, we obtain
\begin{equation}
\left|
\operatorname{Corr}_{m,\varphi}(f,g;t)
\right|
\le
\Big[
2(2\pi)^{-s}N^{-s}
+
C_\alpha A_{d,\alpha}^2
e^{-\gamma_\alpha t}
N^{d+2\alpha}
\Big]
\|f\|_{H^s}
\|g\|_{H^s}.
\label{eq:before-balancing}
\end{equation}

Set $q:=s+d+2\alpha,
N(t):=\left\lceil e^{\frac{\gamma_\alpha t}{q}}\right\rceil.$ 
Since $
e^{\frac{\gamma_\alpha t}{q}}
\le N(t)
\le 2e^{\frac{\gamma_\alpha t}{q}}$, 
we have
\[
N(t)^{-s}
\le e^{-s\frac{\gamma_\alpha t}{q}},
\qquad
e^{-\gamma_\alpha t}N(t)^{d+2\alpha}
\le
2^{d+2\alpha}e^{-s\frac{\gamma_\alpha t}{q}}.
\]
Substituting these estimates into \eqref{eq:before-balancing} gives
\[
\left|
\operatorname{Corr}_{m,\varphi}(f,g;t)
\right|
\le
C_{s,\alpha,d}
e^{-\gamma_{s,\alpha,d}t}
\|f\|_{H^s}
\|g\|_{H^s},
\]
where
\[
\gamma_{s,\alpha,d}
=
\frac{s}{s+d+2\alpha}\gamma_\alpha,
\quad
C_{s,\alpha,d}
=
2(2\pi)^{-s}
+
2^{d+2\alpha}C_\alpha A_{d,\alpha}^{\,2}.
\]
This proves \eqref{eq:sobolev-mixing-conclusion}.
\end{proof}

\paragraph{Acknowledgements}
Nianci Jiang acknowledges support from the Key Laboratory of Random Complex Structures and Data Science, Academy of
Mathematics and Systems Science, Chinese Academy of Sciences, and the National Natural Science Foundation of China
(Grant No.12090010, 12090014, 12471138).
Weili Zhang acknowledges support from the New Cornerstone Investigator Program led by Jinxin Xue.

\paragraph{AI Declaration}
The authors used ChatGPT as an auxiliary tool in this work. The authors verified and completed all mathematical arguments and take full responsibility for the content of this paper.

\end{document}